\documentclass[12pt]{article}
 \expandafter\let\csname equation*\endcsname\relax
 \expandafter\let\csname endequation*\endcsname\relax
\usepackage{bm}
\usepackage{enumerate}
\usepackage[top=1in, bottom=1in, left=1in, right=1in]{geometry}
\usepackage{mathrsfs,color}
\usepackage{amsfonts}\usepackage{multirow}
\usepackage{amssymb}
\usepackage{caption,comment}
\usepackage{blkarray}
\usepackage{amsmath}
\usepackage{float}
\usepackage{amssymb,amsthm}
\usepackage[toc,page]{appendix}
\usepackage{graphicx,afterpage}
\usepackage{dsfont}
\usepackage{mhequ}
\usepackage[disable]{todonotes}

\usepackage{mathdots}
\usepackage{hyperref}
\numberwithin{equation}{section}
\numberwithin{figure}{section}

\newcommand\tabcaption{\def\@captype{table}\caption}

\newtheorem{thm}{Theorem}[section]

\newtheorem{lem}[thm]{Lemma}
\newtheorem{prop}[thm]{Proposition}

\newtheorem{aspt}[thm]{Assumption}
\newtheorem{rem}[thm]{Remark}

\newcommand{\cut}{\varphi}
\newcommand{\CX}{\mathcal{X}}
\newcommand{\Mc}{M_{2}}
\newcommand{\Me}{M_{1}}
\newcommand{\CY}{\mathcal{Y}}
\newcommand{\BX}{\mathbf{X}}
\newcommand{\BY}{\mathbf{Y}}
\newcommand{\CB}{\mathcal{B}}

\newcommand{\unit}{\mathds{1}}

\newcommand{\Vhat}{\widehat{V}}
\newcommand{\Xhat}{\widehat{X}}
\newcommand{\Bhat}{\widehat{B}}
\newcommand{\CD}{\mathcal{D}}
\newcommand{\Yhat}{\widehat{Y}}
\newcommand{\Vbar}{\overline{V}}
\newcommand{\Chat}{\widehat{C}}
\newcommand{\Ctilde}{\widetilde{C}}
\newcommand{\Shat}{\widehat{S}}
\newcommand{\CE}{\mathcal{E}}

\newcommand{\reals}{\mathbb{R}}

\newcommand{\rhom}{\rho_0}
\usepackage{cases}
\newcommand{\Vhatbar}{\overline{\widehat{V}}}

\newcommand{\devk}{\left|H\widehat{V}^{(k)}_n-Z_n^{(k)}\right|^2}
\newcommand{\devsum}{\frac{1}{K}\sum_{k=1}^K\devk}

\newcommand{\BP}{\mathbf{P}}

\newcommand{\Eo}{\Theta_n}
\newcommand{\Cu}{\Xi_n}

\newcommand{\BE}{\mathbb{E}}

\newcommand{\pitilde}{\tilde{\pi}}

\newcommand{\Ubar}{\overline{U}}

\begin{document}
\title{Nonlinear stability of the ensemble Kalman filter with adaptive covariance inflation}
\author{Xin T Tong, Andrew J  Majda  and David Kelly }
\begin{comment}
\address{Department of Mathematics, and Center for Atmosphere Ocean Science, Courant Institute of
Mathematical Sciences, New York University, 251 Mercer Street, New York, NY 10012, USA}
\date{\today}
\eads{\mailto{dtkelly@cims.nyu.edu},
\mailto{jonjon@cims.nyu.edu},
\end{comment}
\maketitle

\abstract{The Ensemble Kalman filter and Ensemble square root filters are data assimilation methods used to combine high dimensional nonlinear models with observed data. These methods have proved to be indispensable tools in science and engineering as they allow computationally cheap, low dimensional ensemble state approximation for extremely high dimensional turbulent forecast models. From a theoretical perspective, these methods are poorly understood, with the exception of a recently established but still incomplete nonlinear stability theory. Moreover, recent numerical and theoretical studies of catastrophic filter divergence have indicated that stability is a genuine mathematical concern and can not be taken for granted in implementation. In this article we propose a simple modification of ensemble based methods which resolves these stability issues entirely. The method involves a new type of adaptive covariance inflation, which comes with minimal additional cost. We develop a complete nonlinear stability theory for the adaptive method, yielding Lyapunov functions and geometric ergodicity under weak assumptions. We present numerical evidence which suggests the adaptive methods have improved accuracy over standard methods and completely eliminate catastrophic filter divergence. This enhanced stability allows for the use of extremely cheap, unstable forecast integrators, which would otherwise lead to widespread filter malfunction.}
 \section{Introduction}
\label{sec:intro}
With the growing importance of accurate weather forecasting and expanding availability of geophysical measurement, data assimilation for high dimensional dynamical system and data has never been more crucial. The ensemble Kalman filter (EnKF) \cite{evensen03} and ensemble square root filters (ESRF) \cite{bishop01,And01} are ensemble based algorithms well designed for this purpose. They quantify the uncertainty of an underlying system using the sample information of a moderate size ensemble $\{V^{(k)}_n\}^K_{k=1}$, thereby significantly reducing the computational cost. The simplicity of these algorithms and their accurate performance has fueled their wide application in various fields of geophysical science \cite{MH12,kalnay03}.
\par
Despite their success, there are many unresolved issues with ensemble based methods. First, there is very little theoretical understanding of the methods and notably the stability framework is incomplete. Only very recently has the stability theory been partially understood in the finite ensemble size scenario, with progress on well-posedness in the fully observed case \cite{KLS14} and subsequently in nonlinear stability of the partially observed case, but under the so-called observable energy criterion \cite{TMK15non}. A better understanding of filter stability is sorely lacking, with recent numerical studies \cite{MH08,GM13} revealing the mechanistically mysterious phenomenon known as \emph{catstrophic filter divergence}, whereby filter state estimates tend to machine infinity, whilst the underlying signal remains in a bounded set. In \cite{KMT15} it has been established rigorously, through an elementary model, that this divergence is not caused by the instability of the numerical integrator alone, instead the update step in the filter itself plays a crucial role in the genesis of divergence.     
\par
In this article we propose a simple modification of EnKF (and ESRF) which resolves these issues completely. The modification is a type of covariance inflation, a widely used strategy for stabilizing and improving the accuracy of filters. Typically the forecast covariance $\Chat_n$ is (in the case of EnKF) additively inflated to $\Chat_n + \lambda_n I$ for some choice of inflation constant $\lambda_n$. Since the forecast covariance decides how much uncertainty is held by the forecast prediction, inflation has the affect of pulling the filter back towards the observations, yielding improved stability. Existing methods of covariance inflation, such as constant additive inflation ($\lambda_n = {\rm{constant}}$) tend to improve accuracy, but are still vulnerable to stability issue like catastrophic filter divergence \cite{MH08}. 
\par
The modification we propose selects the inflation strength $\lambda_n$ \emph{adaptively} and varies according to the distribution of the ensemble. In particular, if the filter is deemed to be performing well, the inflation strength is set to zero and the method is reduced to EnKF (or ESRF, as desired). If the filter is deemed to be `malfunctioning', then the adaptive inflation is \emph{triggered}. The strength of the inflation becomes larger when the filter strays further into malfunction. To decide when and to what extent the filter is malfunctioning, we employ two simple statistics of the ensemble, $\Eo$ and $\Cu$, which are based respectively on the ensemble innovation and the cross correlation between observed and unobserved components. The two statistics are so chosen as it is clear from the theoretical framework that these are precisely the two variables which must be controlled to guarantee stability. Nevertheless, there is a quite natural interpretation as to why these two statistics are an effective gauge of filter performance. The full derivation and explanation of the adaptively inflated methods are given in Section \ref{s:ai}.  
\par
In Sections \ref{s:stab} and \ref{s:ergo} we develop a complete stability theory for the adaptively inflated methods by extending the stability framework established in \cite{TMK15non}. This framework is comprised of two main results: time uniform mean-square estimates on the ensemble via a Lyapunov argument and geometric ergodicity of the signal-ensemble process. We prove that if the underlying model satisfies energy dissipation, then the filter inherits this dissipation and in particular has a Lyapnuov function with compact sub-level sets. This is a vast improvement on the results in \cite{TMK15non} for EnKF, since firstly the observable energy criterion is no longer required and secondly the Lyapunov function for the filter is guaranteed to have compact sub-level sets. This latter fact leads to geometric ergodicity of the signal-ensemble process for the adaptively inflated EnKF, which follows as an immediate corollary of the results in \cite{TMK15non}.    
\par
In Section \ref{s:numerics} we investigate the performance of the adaptively inflated filter numerically, comparing performance with standard methods, such as non-inflated EnKF and non-adaptively (such as constantly) inflated EnKF. The adaptive method performs at least as well and typically better than standard methods and as expected completely avoids the issue of catastrophic filter divergence , which is quite prevalent in all standard methods. The additional expense in computing the inflation strength is minimal and hence the adaptive methods run at almost identical speed to the standard methods. Most impressively, the adaptive method allows for the use of extremely cheap but unstable integrators, like explicit Euler. We will see that such integrators are useless for standard methods due to the prevalence of catastrophic filter divergence, but are very successful when used with adaptive methods.  
\par
Regarding existing literature, there have recently been many methods proposed which employ a type of adaptive inflation for ensemble based methods \cite{And07, And09,BS13,ZH15,YZ15}, but our method is truly novel. Moreover, none of the cited methods have an established theoretical understanding, and our framework may serve as a good starting point for developing such an understanding.     
\par
The structure of the article is as follows. In Section \ref{s:enkf} we define the standard ensemble based methods, EnKF and ESRFs. In Section \ref{s:ai} we define the adaptive inflated modifications of EnKF and ESRFs. In Section \ref{s:stab} we derive time uniform mean-square estimates for the adaptively inflated filters via a Lyapunov function argument. In Section \ref{s:ergo} we prove geometric ergodicity of the signal-ensemble process for the adaptively inflated filters. In Section \ref{s:numerics} we investigate the performance of these filters numerically. In Section \ref{s:conclusion} we conclude with a discussion. The Appendix contains technical tools that will be used for the main results.

\section{Ensemble based methods and adaptive inflation}
\label{s:enkf}
\label{sec:adacov}
%\section{Energy principle for AACI}
\subsection{Model setup}
\label{sec:model}
In this paper, we assume the signal (truth) sequence $U_n\in \mathbb{R}^d$  is generated by\begin{equation}
\label{sys:discrete}
U_{n}=\Psi_h(U_{n-1})+\zeta_n\;,
\end{equation}
where $\Psi_h : \reals^d \to \reals^d$ is a deterministic mapping and $\{\zeta_n\}_{n\geq 1}$ is the system noise. The noise $\zeta_n$ is assumed to be is independent of $\zeta_{1},\ldots \zeta_{n-1}$ when conditioned on the realization of $U_{n-1}$, the conditional mean is zero $\BE (\zeta_n | U_{n-1}) = 0$ and the conditional covariance $\mathbb{E}(\zeta_n\otimes \zeta_n|U_{n-1})$ is denoted by $R_h(U_{n-1})$. Due to the conditional independence of the noise sequence, this gives rise to a Markov chain. 
\par
For example, the model is often generated by the solution of a stochastic differential equation (SDE) 
\begin{equation}
\label{sys:flow}
du_t=\psi(u_t)dt+\Sigma dW_t\;,
\end{equation}
for a sufficiently regular vector field $\psi : \reals^d \to \reals^d$, diffusion coefficient $\Sigma \in \reals^{d \times e}$ and $e$-dimensional Wiener process $W$. We then take $U_n = u_{nh}$ for some fixed $h>0$. In the notation above, we have $\Psi_h(u_0) = \BE (u_h | u_0)$ and $\zeta_n = u_{nh} - \Psi_h(u_{(n-1)h})$. It is easy to see that this satisfies the above conditions.
%
%
%In particular, if we write the transition kernel of process $u_t$ from time $0$ to time $h$ as $K_h(u, dv)$, \eqref{sys:flow} is time discretized into \eqref{sys:discrete} through
%\[
%\Psi_h(u)=\mathbb{E}(u_h|u_0=u)=\int vK_h(u,dv),\quad \zeta_n=u_{nh}-\Psi_h(u_{(n-1)h}).  
%\]

%\subsection{Nonlinear systems with energy principles}

%This satisfies \eqref{sys:discrete} with $\Psi_h(U_n) = \BE (U_n | U_{n-1})$ and $\zeta_n = U_n - \BE (U_n | U_{n-1})$  
%\todo{DK: I cut out the extra stuff about the SDE. XT: good idea. Someone may criticize the concreteness}
% 
%In many practical cases, the nonlinear map $\Psi_h$ and noise sequence $\zeta_n$ can be constructed by stopping 
%  stochastic differential equations (SDE) derived by physical laws at time sequence $h,2h,\ldots$. These SDEs can generally be written as  \todo{$\Sigma$ depends on $U$? }
%\begin{equation}
%\label{sys:flow}
%dU_t=\psi(U_t)dt+\Sigma_UdW_t.
%\end{equation}
%Assume that the transition density of this process initiated from point $u$ at time $h$ is $p_h(u,dx)$, then the nonlinear mapping $\Psi_h$ in \eqref{sys:discrete} can be obtained by the following
%\[
%\Psi_h(u)=\int x p_h(u,dx),\quad \mathbb{P}(\zeta_n\in dx|U_{n-1})=p_h (U_{n-1}, dx-\Psi_h(u)).
%\]
%In other words $\Psi_h(U_{n-1})$ and $\zeta_n$ are the conditional mean and deviation of \eqref{sys:flow} given the initial point is $U_{n-1}$. When $h$ is a very short time interval, the construction above can be approximately seen as a numerical time discretization scheme of \eqref{sys:flow}, but the two are different for general length $h$.  

\subsection{Linear observations}
\label{sec:linobs}
We assume that the truth $U_n$ is observed linearly with a mean zero Gaussian perturbation 
\[
Z_n=HU_n+\xi_n\;.
\] 
where the observation matrix $H$ is of size $q\times d$ with $H = (H_0 , 0_{q\times (d-q)})$ where $H_0 = {\rm{diag}} (h_1, \dots , h_q  )$, $q \leq d$ is the rank of $H$ and  $h_1,\dots,h_q > 0$ are fixed scalars. For the observational noise, we assume that $\mathbb{E}(\xi_n |U_{n-1})=0$ and $\mathbb{E}(\xi_n\otimes \xi_n|U_{n-1})=I_q$.
\par
This seemingly restrictive setting can be assumed without loss of generality. Indeed, any observational noise covariance can be reduced to the identity via a simple rotation on the filtering problem. Suppose that $\xi_n$ has a nonsingular covariance matrix $\Gamma$ (we do not consider the singular case in this article) and  $\Gamma^{-1/2}H$ has an SVD decomposition $\Gamma^{-1/2}H=\Phi \Lambda \Psi^T$, then we rotate the coordinate system and consider 
\begin{equation}
\label{eqn:HSVD}
\widetilde{U}_n=\Psi^TU_n,\quad \tilde{\xi}_n=\Phi^T\Gamma^{-1/2}\xi_n,\quad
\widetilde{Z}_n=\Phi^T \Gamma^{-1/2}Z_n=\Lambda \widetilde{U}_n+\tilde{\xi}_n. 
\end{equation}
Hence this change of coordinates reduces the observation matrix and the observational noise covariance to the desired form.

If the observation dimension $q$ is larger than the model dimension $d$, the last $d-q$ diagonal entries of $\Lambda$ are zero, so the last $d-q$ rows of $\widetilde{Z}_n$ are independent of the signal and play no role in filtering, hence we can ignore and set $d=q$. 
\par
Since all the transformations above are linear and  bijective,  filtering $\widetilde{U}_n$ with $\widetilde{Z}_n$ is equivalent to filtering $U_n$ with $Z_n$, in the sense that the subsequent assumptions and results hold equally for the original and transformed system. When necessary, this will be clarified within the text. 

%On the other hand, if the covariance $\Gamma$ is singular, then certain linear subspace can be observed exactly, and may cause the filtering operation to be singular. We do not consider this scenario in this paper. 
%
%When the observation dimension is strictly less than the model dimension, $q<d$, there are certain directions not observable through $H$. The rotation produced by the normalizing procedure \eqref{eqn:HSVD} actually splits the the model dimension into directly observed direction and non directly observed directions, since $\Lambda$ is a diagonal matrix with first $q$ entries nonzero and last $d-q$ entries being zeros. 
%
%\begin{rem}
%We will see in the sequel that this transformation does not affect the subsequent assumptions, nor the results of stability and geometric ergodicity.   
%\end{rem}
%
%The last remarks of the rotation generated by \eqref{eqn:HSVD} is that it does not affect Assumption \ref{aspt:kinetic}, since square norm is rotation invariant. On the other hand, the deviation statistics used to define covariance inflation in \eqref{eqn:lambda} are not invariant under rotation.  So to apply the results of this article to a filtering problem with non-trivial observational covariance, one must first apply the rotation, then apply AACI to the rotated filter ensemble. 
%\todo{DK: Bad place for this ... }

%This simplification does not sacrifice any generality, since we can always reduce our filtering problem with non degenerate observation noise to this canonical form through the a linear transformation: 

\subsection{Ensemble Kalman filter}\label{s:enkf_algo}

In the standard Kalman filtering theory, the conditional distribution of the signal  process $U_n$  given the observation sequence $Z_1,\ldots, Z_n$ is given by a Gaussian distribution. EnKF inherits this idea by using a group of ensembles $\{V_n^{(k)}\}_{k=1}^K$ to represent this Gaussian distribution, as the mean and covariance can be taken as the ensemble mean and covariance. The EnKF operates very much like a Kalman filter, except its forecast step requires a Monte Carlo simulation due to the nonlinearity of the system. In detail, the EnKF is an iteration of following two steps, with (for instance) $\widehat{V}^{(k)}_0$ being sampled from the equilibrium measure of $U_n$.
\begin{itemize}
\item Forecast step: from the posterior ensemble at time $n-1$, $\{V^{(k)}_{n-1}\}_{k=1}^K$,
a forecast ensemble for time $n$ is generated by 
\begin{equation}
\label{eqn:forecast}
\widehat{V}_{n}^{(k)}=\Psi_h(V_{n-1}^{(k)})+\zeta^{(k)}_n, 
\end{equation}
where $\zeta^{(k)}_n$ are independent samples drawn from the same distribution as $\zeta_n$. 
\item Analysis step: upon receiving the new observation $Z_n$, random perturbations of it are generated by adding $\xi^{(k)}_{n}$:
\[
Z^{(k)}_{n}=Z_{n}+\xi^{(k)}_{n},
\]
where $\xi^{(k)}_n$ are independent samples drawn from the same distribution as $\xi_{n}$. Using the Kalman update rule, each ensemble member is then updated as follow with $\Chat_n$ being the sample covariance of the forecast ensemble:
\begin{align}
\label{eqn:KFupdate}
\begin{split}
V^{(k)}_{n}&=\widehat{V}^{(k)}_{n}-\widehat{C}_{n}H^T(I+H\widehat{C}_nH^T)^{-1}(H \widehat{V}^{(k)}_n-Z^{(k)}_{n})\;\\
&=(I+\Chat_nH^TH)^{-1}\widehat{V}^{(k)}_n+(I+\Chat_nH^TH)^{-1}\Chat_n H^T Z_n^{(k)}\;,
\end{split}
\end{align}
where 
\begin{equ}\label{e:forecast_cov}
\widehat{C}_{n}=\frac{1}{K-1}\sum_{k=1}^K(\widehat{V}^{(k)}_{n}-\overline{\widehat{V}}_{n})\otimes (\widehat{V}^{(k)}_{n}-\overline{\widehat{V}}_{n}) \;, \quad \overline{\widehat{V}}_{n}=\frac{1}{K}\sum^{K}_{k=1}\widehat{V}^{(k)}_n\;.
\end{equ}
%which is the unique minimizer of the quadratic form
%\[
%J(v)=|Z^{(k)}_{n}-Hv|^2+|\widehat{V}^{(k)}_{n}-v|^2_{\widehat{C}_{n}}. 
%\]
\end{itemize}
See \cite{KLS14} for a derivation of EnKF as an approximate Monte Carlo method for sampling the true posterior distribution. In the case of linear model dynamics, it has been rigorously shown \cite{mandel2011convergence,legland10} that EnKF is an approximate Monte Carlo method for the true posterior distribution, with convergence in the $K\to\infty$ limit.

\par
 Based on our description above,  the augmented process $\{U_n, V^{(1)}_n,\ldots, V^{(K)}_n\}$ is a Markov chain. As above, we will denote the natural filtration up to time $n$ as $\mathcal{F}_n=\sigma\{U_m, V^{(1)}_m,\ldots, V^{(K)}_m,m\leq n\}$, and denote the conditional expectation with respect to $\mathcal{F}_n$ as $\mathbb{E}_n$.
\par
In the sequel it will be convenient to split the ensemble into observed and unobserved components. In particular, for each ensemble member we write $V_n^{(k)} = ( X_n^{(k)} , Y_n^{(k)} )$ where $X_n^{(k)} \in \reals^q$ and $Y_n^{(k)} \in {\rm{ker}}(H)$. Similarly, we write $\Vhat_n^{(k)}=(\Xhat_n^{(k)},\Yhat_n^{(k)})$ and for the forecast ensemble covariance

\[
\widehat{C}_n=\begin{bmatrix}
\widehat{C}^X_n, & \widehat{B}_n\\
\widehat{B}^T_n, & \widehat{C}^Y_n
\end{bmatrix},
\]
where 
\begin{equ}
\widehat{C}^X_{n}=\frac{1}{K-1}\sum_{k=1}^K(\widehat{X}^{(k)}_{n}-\overline{\widehat{X}}_{n})\otimes (\widehat{X}^{(k)}_{n}-\overline{\widehat{X}}_{n}) \; \quad \overline{\widehat{X}}_{n}=\frac{1}{K}\sum^{K}_{k=1}\widehat{X}^{(k)}_n
\end{equ}
and so forth the other components. In this notation, the update rule \eqref{eqn:KFupdate} becomes
\begin{equ}\label{e:X_update}
X_n^{(k)} = \widehat{X}_n^{(k)} - (\Chat_n^X)^T H_0^T (I + H_0 \Chat_n^X H_0^T)^{-1} (H_0 \widehat{X}_n^{(k)} - Z_n^{(k)})
\end{equ}
for the observed components and
\begin{equ}\label{e:Y_update}
Y_n^{(k)} = \widehat{Y}_n^{(k)} - (B_n)^T H_0^T (I + H_0 \Chat_n^X H_0^T)^{-1} (H_0 \widehat{X}_n^{(k)} - Z_n^{(k)})
\end{equ}
for the unobserved components.

\subsection{Ensemble square root filters}
\label{sec:ESRFintro}\label{s:esrf}
One drawback of EnKF comes from its usage of artificial noise $\xi_n^{(k)}$, as this introduces unnecessary sampling errors, particularly when the ensemble size is small \cite{evensen04}. The motivation behind the artificial noise is to make the posterior ensemble covariance 
\[
C_{n}:=\frac{1}{K-1}\sum_{k=1}^K(V^{(k)}_{n}-\overline{V}_{n})\otimes (V^{(k)}_{n}-\overline{V}_{n}),\quad  \overline{V}_{n}:=\frac{1}{K}\sum^{K}_{k=1}V^{(k)}_n,
\]
satisfy the covariance update of the standard Kalman filter
\begin{equation}
\label{eqn:postprior}
C_n=\widehat{C}_n-\widehat{C}_nH^T(H^T\widehat{C}_nH+I)^{-1}H\widehat{C}_n,
\end{equation}
when the left hand is averaged over $\xi_n^{(k)}$ \cite{AA99, HWS01, FB07} . 
% it is shown that this equality holds only on upon averaging over $\xi^{(k)}_n$. This is in fact the original motivation for defining EnKF using perturbed observations. One disadvantage of using perturbed observations is the prevalence of sampling errors, particularly when the ensemble size is small \cite{evensen04}.  
%This can be seen as the true motivation for perturbing observations in the first
%Moreover, in order for this equality to hold, the artificial noises $\xi^{(k)}_n$ were added to the real observation $Z_n$. 
%Intuitively, these artificial perturbations would reduce the information the filter can use and hence also the performance. 
ESRFs, including the ensemble transform Kalman filter (ETKF) and the ensemble adjustment Kalman filter (EAKF), aim to resolve this issue by manipulating the posterior spreads to ensure that \eqref{eqn:postprior} holds. Both ETKF and EAKF algorithms are described by the following update steps, with the only difference occurring in the assimilation step for the spread. As with EnKF, the initial ensemble $\{V_{0}^{(k)}\}_{k=1}^K$ is (for instance) sampled from the equilibrium distribution of $U_n$. 
\begin{itemize}
\item Forecast step: identical to EnKF, the forecast ensembles at time $n$ is generated from posterior ensembles at time $n-1$:
\[
\widehat{V}_{n}^{(k)}=\Psi_h(V_{n-1}^{(k)})+\zeta^{(k)}_n.
\]
The forecast ensemble covariance $\widehat{C}_n$ is then computed using \eqref{e:forecast_cov}. 
\item Assimilation step for the mean: upon receiving the new observation $Z_n$, the posterior ensemble mean is updated through
\begin{equation}
\label{sys:ESRFmean}
\overline{V}_n=\overline{\widehat{V}_n}-\widehat{C}_{n}H^T(I+H\widehat{C}_nH^T)^{-1}(H \overline{\widehat{V}}_n-Z_{n}),\quad \overline{\widehat{V}}_{n}=\frac{1}{K}\sum^{K}_{k=1}\widehat{V}^{(k)}_n. 
\end{equation}
\item Assimilation step for the spread: The forecast ensemble spread is given by the $d \times K$ matrix 
\[
\widehat{S}_n = [\widehat{V}^{(1)}_n-\overline{\widehat{V}}_n,\dots , \widehat{V}^{(K)}_n-\overline{\widehat{V}}_n  ]\;.
\]
To update the posterior spread, first find a matrix $T_n\in \mathbb{R}^{d\times d}$ (for ETKF) or $A_n\in \mathbb{R}^{K\times K}$ (for EAKF) such that
\begin{equation}
\label{sys:ESRFspread}
\frac{1}{K-1}T_n\widehat{S}_n \otimes  T_n\widehat{S}_n =\frac{1}{K-1}\widehat{S}_n A_n\otimes  \widehat{S}_n A_n=\widehat{C}_n-\widehat{C}_nH^T(H^T\widehat{C}_nH+I)^{-1}H\widehat{C}_n\;. 
\end{equation}
The posterior spread is updated to $S_n=T_n\widehat{S}_n$ (for ETKF) or $S_n=\widehat{S}_nA_n$ (EAKF), and the ensemble members are updated to
\[
V^{(k)}_n = \overline{V}_n + S_n^{(k)}\;,
\]
where $S_n^{(k)}$ denotes the $k$-th column of the updated spread matrix $S_n$. 
%\[
%[V_n^{(1)},\ldots, V_n^{(K)}]=S_n+\overline{V}_n\otimes \vec{1}_K,
%\]
%where $\vec{1}_K$ denotes the $K$-dim vector with all entries being $1$. 
By construction, the posterior covariance $C_n=(K-1)^{-1}S_n^TS_n$  satisfies \eqref{eqn:postprior}. 
\end{itemize}
At this stage it suffices to know that such $A_n$ and $T_n$ exist, their finer properties play no role in the discussion concerning stability, but we refer the reader to \cite{MH12, TMK15non}. 
%
%they are not very important for our follow up discussion regarding stability, so it suffices for us to know such $A_n$ and $T_n$ exist. 
%Their formulation will become important when we want to study ergodicity in Section \ref{s:ergo}, and a detailed formulation will be given there.

 Based on our description above,  the augmented process $\{U_n, V^{(1)}_n,\ldots, V^{(K)}_n\}$ is again a Markov chain. As in the previous section, we employ the notation $\BE_n$ to denote conditional expectation with respect to $\mathcal{F}_n$.

\section{Adaptive inflation methods}\label{s:aaci_algo} \label{s:ai}

In this section we introduce the adaptive inflation modification of EnKF, as well as the square root filters ETKF and EAKF. We start with the modification of EnKF, which we refer to \emph{EnKF-AI} with the suffix standing for `adaptive inflation'. The nomenclature for the algorithms and their modifications is summarized at the end of the section, in Table \ref{tab:names}. 
\par
The adaptive inflation algorithm is precisely the EnKF algorithm with the forecast covariance $\Chat_n$ replaced by an inflated forecast covariance $\Ctilde_n = \Chat_n + \lambda_n I$. More precisely, the filter ensemble $\{V_n^{(k)}\}_{k\leq K}$  is governed by 
\begin{equation} 
\label{sys:EnKFad}
\begin{gathered}
V^{(k)}_{n}=\widehat{V}^{(k)}_{n}-\widetilde{C}_{n}H^T(I+H\widetilde{C}_nH^T)^{-1}(H \widehat{V}^{(k)}_n-Z^{(k)}_{n})\\
\widehat{V}^{(k)}_{n}=\Psi_h(V^{(k)}_{n-1})+\zeta^{(k)}_n,\quad \overline{\widehat{V}}_{n}=\frac{1}{K}\sum^{K}_{k=1}\widehat{V}^{(k)}_n,\quad Z^{(k)}_{n+1}=Z_{n}+\xi^{(k)}_{n},\\
\widetilde{C}_n=\widehat{C}_{n}+\lambda_n I_d,\quad
\widehat{C}_{n}=\frac{1}{K-1}\sum_{k=1}^K(\widehat{V}^{(k)}_{n}-\overline{\widehat{V}}_{n})\otimes (\widehat{V}^{(k)}_{n}-\overline{\widehat{V}}_{n})\;.
\end{gathered}
\end{equation}

 The inflation parameter $\lambda_n$ is chosen in such a way that it only plays a role when the filter is malfunctioning. In particular, when the filter is giving accurate predictions the inflation parameter will be zero. The inflation will be `triggered' when either of two statistical properties of the filter, $\Eo$ or $\Cu$, exceeds appropriately chosen thresholds, indicating that filter predictions have begun to stray far from the truth. Hence we define 
 \begin{equ}\label{eqn:lambda}\label{e:lambda}
 \lambda_n = \cut (\Eo,\Cu) \;, \quad \cut (x,y) =  c_\cut x (1+y) \unit_{\{x > \Me \; {\rm{or}} \; y >\Mc\}}
 \end{equ}
where $\cut$ plays the role of a cut-off function, $c_\cut$ is some fixed positive constant and $\Me,\Mc$ are fixed positive thresholds constant used to decide whether the filter is functioning properly or not.  
 \par
The first statistical quantity $\Eo$, measures how far predicted observations are from actual observations: 
\[
\Eo:=\sqrt{\devsum}\;.
\] 
In a standard Kalman filter, the quantity $H \Vhat_n - Z_n$ is called the \emph{innovation process} and is used for to test accuracy of the filter. Hence the quantity $\Eo$ should be thought of as the norm of the ensemble innovation process.  
\par
The second statistical property $\Cu$ measures the forecast ensemble covariance between observed components and unobserved components. Using the observed-unobserved notation from Section \ref{s:enkf_algo}, we have $\Cu = \|{\Bhat_n}\|$ where $\Bhat_n$ is the covariation between observed and unobserved components
\[
\widehat{B}_n=\frac{1}{K-1}\sum_{k=1}^K (\Xhat_n^{(k)}-\overline{\Xhat}_n)\otimes (\Yhat_n^{(k)}-\overline{\Yhat}_n)\;.
\] 
Intuitively, it is important to control this cross-variation, since this avoids the situation whereby the filter magnifies a small error in the observed component and imposes it on the unobserved component. This was a key mechanism leading to the filter instability found in \cite{KMT15}.  
\par
Although both statistics $\Eo$ and $\Cu$ can be intuitively connected to the performance of the filter, the real reason we choose to control these two statistics is a mathematical one. If one attempts to prove boundedness results of Section \ref{s:stab} for the unmodified EnKF, it becomes clear that the ensembles cannot be stabilized unless one has adequate control on these two quantities. Hence it is a simple observation that an adaptive inflation which guarantees control of these two statistics will therein guarantee boundedness and stability of the corresponding adaptive filter.    

\subsection{Adaptive inflation for other ensemble based filters}\label{s:other_filters}
The adaptive inflation approach can be applied to other ensemble based filters for guaranteed stability with minimal changes to the original filtering method. 
\par
Two popular modifications to EnKF are constant additive covariance inflation and constant multiplicative covariance inflation. With these two modifications, the analysis step in EnKF operates with the ensemble covariance $\Chat_n$ in \eqref{eqn:KFupdate} replaced by 
\begin{equ}\label{e:enkf_ci}
\Ctilde_n = \Chat_n+\rho I \;, \quad  \Ctilde_n = (1+\rho)\Chat_n\;,
\end{equ}
respectively, and $\rho > 0$ here is a fixed constant. We will refer to EnKF with constant inflation as \emph{EnKF-CI}, with the suffix standing for `constant inflation'. Typically we will use additive rather than multiplicative inflation, but the theoretical results will hold for both verbatim.   
\par

We can combine constant and adaptive inflation by simply adding both simultaneously
\begin{equ}\label{e:enkf_cai}
\Ctilde_n=\Chat_n+\rho I+\lambda_n I,\quad \Ctilde_n=(1+\rho)\Chat_n+\lambda_n I
\end{equ}
respectively in the analysis step \eqref{eqn:KFupdate}. We will refer to this as \emph{EnKF-CAI}, with the suffix standing for `constant+adaptive inflation'. 
\par
In the case of the ESRFs introduced in Section \ref{s:esrf}, we follow the similar pattern of simply replacing $\Chat_n$ with some inflated $\Ctilde_n$. However, this replacement only occurs when calculating the mean and not when calculating the posterior covariance. This is because $C_n$ is the posterior ensemble covariance, its rank cannot exceed $K-1$, while additive inflation can make the right hand side of \eqref{eqn:postprior} of rank $d$, and in most practical cases the model dimension $d$ is larger than the ensemble size $K$. 
\par  
Hence the adaptively inflated version of ESRF is given by  
\begin{equation}
\label{sys:ESQFad}
\begin{gathered}
\overline{V}_{n}=\overline{\widehat{V}}_{n}-\widetilde{C}_{n}H^T(I+H\widetilde{C}_nH^T)^{-1}(H \overline{\widehat{V}}_n-Z_n),\\
C_n=\widehat{C}_n-\widehat{C}_nH^T(I+H\widehat{C}_nH^T)^{-1}H\widehat{C}_n\;.
\end{gathered}
\end{equation}
where
\begin{equs} \label{e:esrf_lambda} \label{e:esrf_inflation}
\begin{gathered}
\Ctilde_n = \Chat_n + \lambda_n I \quad \text{(ETKF-AI / EAKF-AI)},\\
\Ctilde_n = \Chat_n + \rho I \quad \text{(ETKF-CI / EAKF-CI)},\\
\Ctilde_n = \Chat_n + \rho I + \lambda_n I \quad \text{(ETKF-CAI / EAKF-CAI)}.
\end{gathered}
\end{equs}
\todo{DK: Do we need multiplicative here?}
Notice that we use the same nomenclature from Table \ref{tab:names} as in the case of EnKF. The adaptive inflation strength $\lambda_n$ defined the same as EnKF, but with $\Eo$ now defined as 
\begin{equ}
\Eo := \frac{1}{K}\sum_{k=1}^K |H\Vhat^{(k)}_n-Z_n|^2\;,
\end{equ}
since $Z_n^{(k)}$ no longer plays a role in the setting of ESRF.

\begin{table}[h!]
\begin{center}
\begin{tabular}{|c | c| c|}
\hline
Acronym & Algorithm & Equation reference \\ 
\hline
EnKF & Ensemble Kalman filter & \eqref{eqn:KFupdate} \\
EnKF-AI & Ensemble Kalman filter w/ adaptive inflation & \eqref{sys:EnKFad} \\
EnKF-CI & Ensemble Kalman filter w/ constant inflation & \eqref{e:enkf_ci} \\ 
EnKF-CAI & Ensemble Kalman filter w/ constant + adaptive inflation & \eqref{e:enkf_cai} \\
%\hline
%ETKF & Ensemble transform Kalman filter & \eqref{e:} \\
%ETKF-AI & Ensemble transform Kalman filter w/ adaptive inflation & \eqref{e:} \\
%ETKF-CI & Ensemble transform Kalman filter w/ constant inflation & \eqref{e:} \\ 
%ETKF-CAI & Ensemble transform Kalman filter w/ constant + adaptive inflation & \eqref{e:} \\
%\hline
%EAKF & Ensemble adjustment Kalman filter & \eqref{e:} \\
%EAKF-AI & Ensemble adjustment Kalman filter w/ adaptive inflation & \eqref{e:} \\
%EAKF-CI & Ensemble adjustment Kalman filter w/ constant inflation & \eqref{e:} \\ 
%EAKF-CAI & Ensemble adjustment Kalman filter w/ constant + adaptive inflation & \eqref{e:} \\
\hline
\end{tabular}
\end{center}
\caption{ Acronyms for the ensemble methods. }
\label{tab:names}
\end{table}

\section{Energy principles}\label{s:stab}
\label{sec:energy}

In this section we show that the adaptive filtering methods introduced in Section \ref{s:ai} inherit an energy principle from the underlying model. Using the notation for the model introduced in Section \ref{s:enkf}, the energy principle for the model takes the following form:
\begin{aspt}[Energy principle]\label{ass:energy}
\label{aspt:kinetic} 
There exist constants $\beta_h \in (0,1), K_h > 0$, such that
\[
\mathbb{E}_{n-1}|U_n|^2 \leq (1-\beta_h) |U_{n-1}|^2+K_h \quad \text{a.s.} \;,
\]
recalling that $\mathbb{E}_{n-1}$ denotes conditional expectation with respect to the $\sigma$-algebra $\mathcal{F}_{n-1} : = \sigma (U_0,U_1,\dots,U_{n-1})$. Equivalently,
\[
|\Psi_h(u)|^2+ \text{tr}(R_h(u))\leq (1-\beta_h) |u|^2+K_h\;,
\]
for all for all $u\in \mathbb{R}^d$. 
\end{aspt} 
\begin{rem}
There is a slight abuse of notation here, since we also use $\mathbb{E}_{n-1}$ to denote expectation conditioned on the history of the signal-ensemble process. However it is clear that in the above situation the two coincide. 
\end{rem}

When $U_n$ is given by the discrete time formulation of an SDE as in \eqref{sys:flow}, it follows from a simple Gronwall argument that Assumption \ref{aspt:kinetic} holds as long as $\langle \psi(u), u\rangle\leq -\beta |u|^2+k_h$ with constants $\beta, k_h>0$. 
%Because then the generator of $|u|^2$ satisfies 
%\[
%\mathcal{L}|u|^2=-2\beta|u|^2+2k_h+\text{tr}(\Sigma\Sigma^T),
%\]
%which by Gr\"{o}nwall's inequality and Dynkin's formula yields
%\begin{align*}
%\mathbb{E} |u_h|^2&\leq e^{-2\gamma h}|u_0|^2+\int^h_0 e^{-2\gamma (h-s)}(2k+\text{tr}(\Sigma\Sigma^T)) ds\\
%&\leq e^{-2\gamma h}|u_0|^2+h(2k+\text{tr}(\Sigma\Sigma^T)).
%\end{align*}
Therefore, the discrete time formulation of truncated Navier Stokes equation,  Lorenz 63 and 96 models all satisfy Assumption \ref{aspt:kinetic}, explicit verifications can be found in Section 2.5 of \cite{TMK15non}. 
\par
It is natural to ask whether a filter inherits an energy principle from its underlying model and hence inherits a type of stability. In \cite{TMK15non} it was shown that the non-adaptive ensemble based filters (EnKF, ESRF) inherit a stronger notion of energy principle called the observable energy principle. In particular, if the model satisfies
\begin{equ}\label{e:stab_obs}
|H \Psi_h(u)|^2+ \text{tr}(H R_h(u)H^T)\leq (1-\beta_h) |H u|^2+K_h\;,
\end{equ}
then the filter satisfies a related energy principle, which a Lyapunov function whose sublevel sets are only compact in the case where $H$ is of full rank (complete observations). This is clearly a much weaker result than one would hope for, since one must assume \eqref{e:stab_obs} which is not thought to be true for partially observing filters and moreover the result is not as strong as desired since the Lyapunov function cannot be used for ergodicity arguments.  
\par
We will now show that the adaptive inflation filters completely avoid these issues. In particular, the adaptively inflated filters inherit the usual energy principle (Assumption \ref{ass:energy}) and moreover, the Lyapnuov function for the ensemble has compact sub-level sets for any choice of observation matrix. The inflation mechanism additionally ensures that the observed directions of the filter $HV_n^k$ never stray from the (perturbed) observations $Z_n^k$, this is indeed the main purpose of the inflation.

\begin{thm}
\label{thm:adEnKF}
Let $\{V_n^{(k)}\}_{k=1}^K$ denote the EnKF-AI ensemble described by \eqref{sys:EnKFad} with an inflation strength $\lambda_n \geq \varphi(\Eo,\Cu)$, then we have the following estimates:  
\begin{enumerate}[(i)]
\item The ensemble innovation satisfies the almost sure bound
\[
|HV^{(k)}_n-Z^{(k)}_n|\leq  \sqrt{K} \max\{\Me, \rhom^{-1}c_\cut^{-1}\}
\]
Here $\rhom$ denotes the minimum eigenvalue of $H_0H_0^T$. 
%\item The total deviation of the posterior ensemble from its mean is bounded on average,
%\[
%\sum_{k=1}^K\mathbb{E}_{n-1}|V^{(k)}_n-\overline{V}_n|^2
%\leq 2\rho_m^{-1} Kd\quad \text{a.s.}
%\]
\item Suppose the model $U_n$ additionally satisfies an energy principle (Assumption \ref{aspt:kinetic}) and let
\[
\CE_n=4K\rho_0^{-1}\beta_h^{-1}\|H\|^2|U_n|^2+\sum_{k=1}^K |V^{(k)}_n|^2\;.
\]
Then there exists constant $\beta_h \in (0,1)$, $D > 0$ such that
\[
\mathbb{E}_{n-1}\CE_n\leq (1-\tfrac{1}{2}\beta_h) \CE_{n-1}+D\;.
\]
In particular, we have the time uniform bound $\sup_{n\geq 0} \mathbb{E} \CE_n < \infty$. 

\end{enumerate}
\end{thm}

\begin{rem}
It will be clear from the proof that the above result also holds verbatim for EnKF-CAI, introduced in \eqref{e:enkf_cai}. 
\end{rem}

%
%\begin{thm} 
%\label{thm:AACIunobs}
%When filtering nonlinear systems with energy principles,  Assumption \ref{aspt:kinetic}, the EnKF ensemble with AACI \eqref{sys:unobscontrol} using inflation strength $\lambda=\cut(\Eo,\Cu)$ can be  bounded uniformly in time. In particular, the following joint energy of ensemble and unfiltered process
%\[
%\CE_n=4K\rho_m^{-1}\beta_h^{-1}\|H_0\|^2|U_n|^2+\sum_{k=1}^K |V^{(k)}_n|^2,
%\]
% is dissipative with a proper constant $D$
%\[
%\mathbb{E}_{n-1}\CE_n\leq (1-\tfrac{1}{2}\beta_h) \CE_{n-1}+D. 
%\]
%Here $\rhom$ denotes the minimum eigenvalue of  $H_0H_0^T$. 
%\end{thm}

A similar energy principle also holds for the adaptive inflation modifications of ESRF. Note that only the second part of the result holds for the ESRF case.  

\begin{thm}
\label{thm:adESRF}
Let $\{V_n^{(k)}\}_{k=1}^K$ denote the ETKF-AI / EAKF-AI ensemble described by \eqref{sys:ESQFad} , \eqref{e:esrf_inflation} with any inflation strength $\lambda_n \geq \varphi(\Eo,\Cu)$, then Theorem \ref{thm:adEnKF} (ii) holds. 
\end{thm}

\begin{rem}
As above, it will be clear from the proof that the above result also holds verbatim for ETKF-CAI / EAKF-CAI introduced in \eqref{e:esrf_inflation}. 
\end{rem}

We now provide the proof of Theorem \ref{thm:adEnKF}, the proof Theorem \ref{thm:adESRF} is similar and is therefore deferred until the appendix. 

%Both criterions in Section \ref{sec:choice} can be applied here. The only difference is that $\xi^{(k)}_n$ will vanish. So the values for $\sigma^2_{\Eo}$ will be slightly different, in the sense that the $2d$ terms are replaced by $d$. Namely, we have
%\[
%\sigma_{\Eo, E}^2=\|H\|^2\beta_h^{-1}K_h+d,\quad \sigma_{\Eo,eq}^2= \|H\|^2\text{tr}(\text{cov}(U_n))+d,\]
%\[
%\sigma^2_{\Eo,A}=\exp(2\beta_Lh)\|H\|^2\left(\normalfont{\text{Error}_{A}}+2Q h\right)+d. 
%\]
%We omit the proof here because they are completely identical. 
%

\begin{proof}[Proof for Theorem \ref{thm:adEnKF}] Before proceeding, recall the `observed-unobserved' notation from Section \ref{s:enkf_algo}:
\[
H = (H_0 , 0)\;, \quad V_n^{(k)}=(X_n^{(k)},Y_n^{(k)})\;,\quad {\Ctilde}_n=\begin{bmatrix}
{\Ctilde}^X_n, & \widetilde{B}_n\\
\widetilde{B}^T_n, & \widetilde{C}^Y_n
\end{bmatrix},
\]
\[
\Vhat_n^{(k)}=(\Xhat_n^{(k)},\Yhat_n^{(k)}) \;, \quad \widehat{C}_n=\begin{bmatrix}
\widehat{C}^X_n, & \widehat{B}_n\\
\widehat{B}^T_n, & \widehat{C}^Y_n\;.
\end{bmatrix}
\]
We start with part (i). From the assimilation equation \eqref{sys:EnKFad}, by applying $H$ to both sides, rearranging and applying the identity $H\Ctilde_n H^T = H_0 \Ctilde_n^X H_0^T$ we obtain
\begin{equation}
\label{tmp:vz}
HV^{(k)}_n-Z^{(k)}_n
=(I+H_0\widetilde{C}^X_nH_0^T)^{-1}(H\widehat{V}^{(k)}_n-Z^{(k)}_n). 
\end{equation}
Since the minimum eigenvalue of matrix $H_0H_0^T$ is $\rho_0>0$ and $H_0\widehat{C}^X_n H_0^T$ is positive semi-definite, we have
\[
I+H_0\widetilde{C}_n^XH_0^T=I+H_0\widehat{C}^X_n H_0^T+\cut(\Eo,\Cu) H_0H_0^T
\succeq (1+\rho_{0}\cut(\Eo,\Cu)) I_d\;.
\]
Therefore, when $\Eo\leq \Me$, we have
\[
|(I+H_0\widetilde{C}^X_nH_0^T)^{-1}(H\Vhat^{(k)}_n-Z^{(k)}_n)|
\leq  \frac{|H\Vhat^{(k)}_n-Z^{(k)}_n|}{1+\rho_{0}\cut(\Eo,\Cu)} \leq {|H\Vhat^{(k)}_n-Z^{(k)}_n|} \leq {\sqrt{K}}\Me\;. 
\]
On the other hand, when $\Eo> \Me$, we have
\begin{equs}
|(I+H_0\widetilde{C}_n^XH_0^T)^{-1}(H\Vhat^{(k)}_n-Z^{(k)}_n)| 
 \leq \frac{|H\Vhat^{(k)}_n-Z^{(k)}_n|}{1+\rho_{0}\cut(\Eo,\Cu)} &= \frac{|H\Vhat^{(k)}_n-Z^{(k)}_n|}{\rho_{0}c_\varphi \Eo (1 + \Cu) }  \\ &\leq \frac{|H\Vhat^{(k)}_n-Z^{(k)}_n|}{\rho_{0}c_\varphi \Eo } \\ 
& \leq \sqrt{K}\rho_0^{-1}c_\cut^{-1}. 
\end{equs}
This bound and \eqref{tmp:vz} yield claim (i).  
%Notice the deviation of $k$-th ensemble from the mean satisfies: 
%\[
%H(V^{(k)}_n-\overline{V}_n)
%=(I+H\widetilde{C}_nH^T)^{-1}H(\widehat{V}^{(k)}_n-\overline{\widehat{V}}_n)+
%H\widetilde{C}_nH^T(I+H\widetilde{C}_nH^T)^{-1}(\xi^{(k)}_n-\bar{\xi}_n). 
%\]
%So Young's inequality yields:
%\[
%|H(V^{(k)}_n-\overline{V}_n)|^2
%\leq 2|(I+H\widetilde{C}_nH^T)^{-1}H(\widehat{V}^{(k)}_n-\overline{\widehat{V}}_n)|^2+
%2|H\widetilde{C}_nH^T(I+H\widetilde{C}_nH^T)^{-1}(\xi^{(k)}_n-\bar{\xi}_n)|^2
%\]
%To bound the first term above, notice the sum of  it over all $k$ is
%\begin{align*}
%\sum_{k=1}^K|(I+H\widetilde{C}_nH^T)^{-1}H(\widehat{V}^{(k)}_n-\overline{\widehat{V}}_n)|^2
%&=\sum_{k=1}^K\text{tr}((I+H\widetilde{C}_nH^T)^{-2}H(\widehat{V}^{(k)}_n-\overline{\widehat{V}}_n)\otimes H(\widehat{V}^{(k)}_n-\overline{\widehat{V}}_n) )\\
%&=\text{tr}((I+H\widetilde{C}_nH^T)^{-2} H\widehat{C}_nH^T).
%\end{align*}
%Note that $(I+H\widetilde{C}_nH^T)^2\succeq H\widehat{C}_nH^T$, by elementary Lemma \ref{lem:trinv}, the quantity above is bounded by $d$ a.s. The second term in the Young's inequality above, is bounded by 
%\[
%\mathbb{E}_{n-1}|H\widetilde{C}_nH^T(I+H\widetilde{C}_nH^T)^{-1}(\xi^{(k)}_n-\bar{\xi}_n)|^2
%\leq \mathbb{E}_{n-1}|\xi^{(k)}_n-\bar{\xi}_n|^2=\frac{K-1}{K}d. 
%\]
\par
We now move on to part (ii), starting by bounding the observed part. Starting from the expression \eqref{e:X_update} and applying part (i), Young's inequality and the independence of $\xi_n^{(k)}$, we obtain
\begin{align}
\notag
\mathbb{E}_{n-1}|X^{(k)}_n|^2&\leq 2\mathbb{E}_{n-1}\left(|X^{(k)}_n-H_0^{-1}Z_n^{(k)}|^2+|H^{-1}Z_n^{(k)}|^2\right)\\
\notag
&\leq 2\rhom^{-1}\mathbb{E}_{n-1}(|HV^{(k)}_n-Z_n^{(k)}|^2+|Z_n^{(k)}|^2)\\
\notag
&\leq 2\rhom^{-1}K\max\{\Me^2,\rhom^{-2}c_\cut^{-2}\}+2\rhom^{-1}\|H\|^2\mathbb{E}_{n-1}|U_n|^2
+2\rhom^{-1}\mathbb{E}_{n-1}|\xi_n^{(k)}+\xi_n|^2\\
\label{tmp:xk}
&=2\rhom^{-1}K\max\{\Me^2,\rhom^{-2}c_\cut^{-2}\}+2\rhom^{-1}\|H\|^2\mathbb{E}_{n-1}|U_n|^2
+4\rhom^{-1}d.
\end{align}
%Summation over all $k$ yields:
%\[
%\sum_{k=1}^K\mathbb{E}_{n-1}|X^{(k)}_n|^2\leq 2K\rhom^{-1}\|H\|^2\mathbb{E}_{n-1}|U_n|^2
%+ 2K\rhom^{-1}\max\{\Me^2,K\rhom^{-2}c_\cut^{-2}\}+4K\rhom^{-1}d.
%\]
For the unobserved part, starting from \eqref{e:Y_update} and applying Young's inequality (Lemma \ref{lem:young}) 
\begin{equation}
\label{tmp:Yj2}
|Y^{(k)}_n|^2\leq (1+\frac{1}{2}\beta_h)|\widehat{Y}^{(k)}_n|^2+(1+2\beta_h^{-1})|\widehat{B}^T_nH_0^T(I+H_0\widetilde{C}^X_nH_0^T)^{-1}(H_0\widehat{X}^{(k)}_n-Z^{(k)}_n)|^2\;,
\end{equation}
where $\beta_h \in (0,1)$ is the one appearing in Assumption \ref{ass:energy}. For the first term on the right hand side, using Assumption \ref{aspt:kinetic} and the elementary inequality $(1+\frac{1}{2}\beta_h)(1-\beta_h)\leq (1-\frac{1}{2}\beta_h)$ we have
\[
\mathbb{E}_{n-1}(1+\frac{1}{2}\beta_h)|\widehat{Y}_n^{(j)}|^2
\leq (1+\frac{1}{2}\beta_h)\mathbb{E}_{n-1}|\Vhat_n^{(j)}|^2
\leq (1-\frac{1}{2}\beta_h) |V_{n-1}^{(j)}|^2+(1+\frac{1}{2}\beta_h)K_h\;.
\]
Since $\Cu=\|\widehat{B}_n\|$, the second part of \eqref{tmp:Yj2} is bounded by
\begin{equs}
|\widehat{B}^T_nH_0^T(I+H_0\widetilde{C}^X_nH_0^T)^{-1}(H_0\widehat{X}^{(j)}_n-Z^{(j)}_n)|^2
&\leq\Cu^2\|H_0\|^2\|(I+H_0\widetilde{C}^X_nH_0^T)^{-1}\|^2 |H_0\widehat{X}^{(j)}_n-Z^{(j)}_n|^2 \\ 
&\leq K \Eo^2 \Cu^2\|H_0\|^2\|(I+H_0\widetilde{C}^X_nH_0^T)^{-1}\|^2 
\end{equs}
where in the last inequality we used the bound $|H_0\widehat{X}^{(j)}_n-Z^{(j)}_n|^2\leq K\Eo^2$. 
%
%\label{tmp:bh}
%&\quad\leq K\|(I+H_0\widetilde{C}^X_nH^T_0)^{-1}\|\Eo\Cu. 
%\end{align}
Notice that, as in the proof of part (i)
\[
I+H_0\Ctilde^X_nH_0^T
\succeq (1+\lambda_n\rhom) I_q\;.
\]
%Lemma \ref{lem:trinv} leads to \todo{DK: Fix $\varphi$ notation}
%\[
%\|(I+H_0\Ctilde^X_nH_0^T)^{-1}\|\leq (1+\rhom \cut(\Eo,\Cu))^{-1}.
%\]
Hence, when $\Eo>\Me$ or $\Cu>\Mc$ we have
\[
\|(I+H_0\widetilde{C}^X_nH^T_0)^{-1}\|\Eo\Cu \leq \frac{\Eo\Cu}{1+\rhom c_\cut \Eo (\Cu+1))}\leq \rhom^{-1} c_\cut^{-1}\;,
\]
and when $\Eo\leq \Me$ and $\Cu\leq \Mc$, 
\[
\|(I+H_0\widetilde{C}^X_nH^T_0)^{-1}\|\Eo\Cu \leq \Eo\Cu\leq \Me \Mc\;.
\]
In conclusion, the following holds almost surely 
\begin{equation}
\label{tmp:bh}
|\widehat{B}^T_nH_0^T(I+H_0\widetilde{C}^X_nH_0^T)^{-1}(H_0\widehat{X}^{(j)}_n-Z^{(j)}_n)|^2
\leq  K\|H_0\|^2\max\{\Me^2\Mc^2, \rhom^{-2}c_\cut^{-2}\}\;.
\end{equation}
By combining the above estimates with \eqref{tmp:Yj2}, the sum of \eqref{tmp:xk} and \eqref{tmp:Yj2} over all $k$ yields the following inequality with a constant  $D_2>0$
\begin{equation}
\label{tmp:vu2}
\mathbb{E}_{n-1}\sum_{k=1}^K|V_n^{(k)}|^2
-2K\rhom^{-1}\|H_0\|^2\mathbb{E}_{n-1}|U_n|^2\leq (1-\frac{1}{2}\beta_h)\sum_{k=1}^K|V_{n-1}^{(k)}|^2+D_2. 
\end{equation}
Then notice a multiple of Assumption \ref{aspt:kinetic} is 
\[
4K\beta_h^{-1}\rho_0^{-1}\|H\|^2\mathbb{E}_{n-1}|U_n|^2\leq 4K\beta_h^{-1}\rho_0^{-1}\|H\|^2(1-\beta_h)|U_{n-1}|^2+4K\beta_h^{-1}\rho_0^{-1}\|H\|^2K_h. 
\]
The sum of the two previous two inequalities yields
\begin{align*}
\mathbb{E}_{n-1}\CE_n&=\sum_{k=1}^K\mathbb{E}_{n-1}|V^{(k)}_n|^2+
4K\beta_h^{-1}\rho_0^{-1}(1-\frac{1}{2}\beta_h)\|H\|^2\mathbb{E}_{n-1}|U_n|^2\\
&=\sum_{k=1}^K\mathbb{E}_{n-1}|V^{(k)}_n|^2-2K\rho_0^{-1}\|H\|^2\mathbb{E}_{n-1}|U_n|^2+4K\beta_h^{-1}\rho_0^{-1}\|H\|^2\mathbb{E}_{n-1}|U_n|^2\\
&\leq {(1-\frac{1}{2}\beta_h)\sum_{k=1}^K|V_{n-1}^{(k)}|^2} + 4K\beta_h^{-1}\rho_0^{-1}\|H\|^2(1-\beta_h)|U_{n-1}|^2+D \\
& \leq {(1-\frac{1}{2}\beta_h)\sum_{k=1}^K|V_{n-1}^{(k)}|^2} + 4K\beta_h^{-1}\rho_0^{-1}\|H\|^2(1-\frac{1}{2}\beta_h)^2|U_{n-1}|^2+D = (1-\frac{1}{2}\beta_h) \CE_{n_1} + D_2\;,
\end{align*}
with $D:=4K\beta_h^{-1}\rho_0^{-1}\|H\|^2K_h+D_2$ and using the fact that $(1-\beta_h) \leq (1-\frac{1}{2}\beta_h)^2  $.
%\[
%\mathbb{E}_{n-1}\CE_n\leq 4K\beta_h^{-1}\rho_0^{-1}\|H\|^2(1-\frac{1}{2}\beta_h)^2|U_{n-1}|^2+D
%\leq (1-\frac{1}{2}\beta_h)\CE_{n-1}+D. 
%\]
\end{proof}

\section{Geometric ergodicity}\label{s:ergo}
A stochastic process is geometrically ergodic if it converges to a unique invariant measure geometrically fast. In this section we will show that the signal-ensemble process (for each filter introduced in this article) is geometrically ergodic. In particular, if $P$ denotes the Markov transition kernel for the signal ensemble process, then we will show that there exists a constant $\gamma \in (0,1)$ such that
\begin{equation}
\|P^n \mu - P^n \nu\|_{TV} \leq C_{\mu,\nu}\gamma^n\;,
\end{equation}
where $\mu,\nu$ are two arbitrary initial probability distributions, $C_{\mu,\nu}$ is a time uniform constant that depends on $\mu,\nu$, and $\|\cdot\|_{TV}$ denotes the total variation norm. 
\par
Geometric ergodicity is a notion of stability for the filter. In particular it implies that discrepancies in the initialization of the filter will dissipate exponentially quickly. 
%In terms of filtering, it implies that discrepancies in the initialization of the ensemble filters and numerical errors occur at each iteration, which are usually inevitable in practice, will dissipate exponentially with time. It also indicates one  long enough trajectory of the filter contains sufficient information of the filter process through Birkhoff ergodic theorem. 
For the linear Kalman filter and optimal filters, it is known that under mild assumptions geometric ergodicity is extended from model to filter \cite{Jaz72, TvH12}. 
\par
In theorems 5.5, 5.6 and 5.8 of  \cite{TMK15non}, the authors proved that if the system noise $\zeta_n$ is non-degenerate and the signal-ensemble process has an energy principle,  then EnKF, ETKF and a natural version of EAKF will be geometrically ergodic. In particular, non degeneracy for the system noise $\zeta_n$ means the following. 
\begin{aspt}[Nondegenerate system noise]
\label{aspt:density}
For any constants $R_1, R_2>0$, there is a constant $\alpha>0$ such that 
\[
\mathbb{P}(\zeta_n\in \cdot \; |U_{n-1}=u)\geq \alpha\lambda_{R_2}(\cdot)
\]
for all $|u|\leq R_1$, where $\lambda_{R_2}(dx)$ is the Lebesgue measure of $\mathbb{R}^d$ restricted to $\{u: |u|\leq R_2\}$.
\end{aspt}
Assumption \ref{aspt:density} holds for many practical examples. When $U_n$ is produced by time discretization of an SDE \eqref{sys:flow}, it suffices to require $\Sigma$ being nonsingular, see  Appendix B in \cite{TMK15non} for a detailed discussion. 
%Assumption \ref{aspt:density} is also satisfied in situations where $U_n$ is produced genuinely as a random sequence and $\zeta_n$ is usually a sequence of non-degenerate random Gaussian variables.
\par
In contrast with the ergodicity results of \cite{TMK15non}, we no longer require the assumption of a Lyapunov function with compact sub-level sets, since this is guaranteed by Theorems \ref{thm:adEnKF} and \ref{thm:adESRF}. This is a vast improvement as in the non-adaptive case such Lyapunov functions are only known to exist in the (essentially) fully observed scenario. 
\par
Before stating the result, recall that $\Me, \Mc > 0$ are the thresholds used to decide when to trigger the adaptive inflation mechanism. For EnKF-AI and ETKF-AI, the thresholds have no constraints besides being positive, but for EAKF-AI we require that the constants are sufficiently large.  
%
%Because Theorems \ref{thm:adEnKF} and \ref{thm:adESRF} above have provided energy principles for the signal-ensemble process, using the methods in \cite{TMK15non} we naturally arrive to the following result 
\begin{thm}
\label{thm:geometric}
If the system noise in the signal process is non-degenerate (Assumption \ref{aspt:density}) then the signal-ensemble processes $(U_n, V^{(1)}_n,\ldots, V^{(K)}_n)$ generated by EnKF-AI (with any $\Me,\Mc > 0$), ETKF-AI (with any $\Me,\Mc > 0$) or a version of EAKF-AI (with $\Me,\Mc$ sufficiently large) are geometrically ergodic. 
\end{thm}
\begin{rem} The same result holds verbatim for adaptive filters with additional constant covariance inflation (EnKF-CAI,ETKF-CAI,EAKF-CAI), which are defined in Section \ref{s:other_filters}. 
\end{rem}

As a matter of fact, the above result is an immediate corollary of \cite{TMK15non}. The reasoning will become clear after we review the major steps in \cite{TMK15non}. 

\subsection{A controllability framework}
The central piece of the arguments is a result of Markov chain theory \cite{MT93}. Here we use a simple adaptation of the form given in \cite[Theorem 2.3]{MS02}.
\begin{thm}
\label{thm:MS02}
Let $\BX_n$ be a Markov chain in a space $E$ such that
\begin{enumerate}
\item  (Lyapunov function) There is a function $\mathcal{E}:E\mapsto \mathbb{R}^+$ with compact sub-level sets and such that
\[
\mathbb{E}_{n-1}\mathcal{E}(\BX_n)\leq (1-\beta)\mathcal{E}(\BX_{n-1})+K\;, 
\]
for some $\beta \in (0,1)$ and $K > 0$\;.
\item   (Minorization condition) Let $C_R=\{x:\mathcal{E}(x)\leq R\}$. For all $R>0$, there exists a probability measure $\nu$ on $E$ with $\nu(C_R)=1$ and a constant $\eta>0$ such that for any measurable set $A \subset E$ 
\[
\mathbb{P}(\BX_n\in A|\BX_{n-1}=x)\geq \eta\nu(A)
\]
for all $x\in C_R$. 
\end{enumerate}
Then there is a unique invariant measure $\pi$ and constants $r \in (0,1), \kappa>0$ such that
\[
\|\mathbb{P}^\mu(\BX_n\in \,\cdot\,)-\pi\|_{TV}\leq \kappa r^n \bigg(1+\int \CE(x) \mu(dx)\bigg). 
\]
\end{thm}
Theorems \ref{thm:adEnKF} and \ref{thm:adESRF} have already provided the Lyapunov function. All that remains is to show the minorization condition. 
\par
To obtain minorization, we will exploit the update structure of the ensemble filters, precisely as was done in \cite{TMK15non}. Notice that for all ensemble filters, the signal-ensemble process $\BX_n := (U_n , V_n^{(1)},\dots,V_n^{(K)})$ is a Markov chain taking values in $\CX := \reals^{d }\times \reals^{d \times K}$. The evolution of $\BX_n$ is described by the composition of two maps. The first is a random map from $\CX$ to a signal-forecast-observation space $\CY$, described by a Markov kernel $\Phi : \CX \times \CB(\CY) \to [0,1]$. The second is a deterministic map $\Gamma : \CY \to \CX$, which combines the forecast with the observed data to produce the updated posterior ensemble. The details of these maps, as well as the definition of the intermediate space $\CY$, differs between EnKF, ETKF, EAKF and their adaptive counterparts.  
\par
For example, in EnKF,  the intermediate space is $\CY := \reals^d \times \reals^{d\times K } \times \reals^{q \times K}$ and the random mapping is 
\[
(U_{n-1} , V_{n-1}^{(1)},\dots,V_{n-1}^{(K)} ) \mapsto \BY_n :=  (U_n, \Vhat_n^{(1)},\dots,\Vhat_n^{(K)},  Z_n^{(1)},\dots,Z_n^{(K)} )\;.
\]
The deterministic map $\Gamma$ is given by the Kalman update   
\[
\Gamma (U_n , \Vhat_n^{(1)},\dots,\Vhat_n^{(K)}, Z_n^{(1)},\dots,Z_n^{(K)}) = (U_n , \Gamma^{(1)},\dots,\Gamma^{(K)})\;,
\]
where 
\begin{equation}\label{e:Gamma_enkf}
\begin{gathered}
\Gamma^{(k)} = \widehat{V}^{(k)}-\Ctilde H^T(I+H\widehat{C}H^T)^{-1}(H \widehat{V}^{(k)}-Z^{(k)})\;, \\
\widetilde{C} = \frac{1}{K-1}\sum_{k= 1}^{K}(\widehat{V}^{(k)}-\overline{\widehat{V}})\otimes (\widehat{V}^{(k)}-\overline{\widehat{V}})\;.
\end{gathered}
\end{equation}
For EnKF-AI, we have the same formulas, but with the exception
\begin{equ}
 \widetilde{C} = \frac{1}{K-1}\sum_{k= 1}^{K}(\widehat{V}^{(k)}-\overline{\widehat{V}})\otimes (\widehat{V}^{(k)}-\overline{\widehat{V}}) + c_\varphi\Eo (1+\Cu)\unit_{\Eo > \Me {\rm \;or\;} \Cu > \Mc} I \;.
\end{equ}
and 
\begin{equ}
\Eo = \sqrt{\devsum}\;, \quad \Cu = \left\|\frac{1}{K-1}\sum_{k=1}^K (\Xhat_n^{(k)}-\overline{\Xhat}_n)\otimes (\Yhat_n^{(k)}-\overline{\Yhat}_n) \right\|
\end{equ}
The corresponding formulas for ETKF, EAKF (and their adaptive counterparts) can be similarly derived and for concreteness are given in Appendix \ref{app:ergo_formulas}. Notice that for values of $\BY_n$ close to the origin, the adaptive inflation term will vanish and hence the deterministic map $\Gamma$ for EnKF-AI coincides with that of EnKF near the origin. Similar statements hold for ETKF-AI and EAKF-AI. 
\par
Given this formulation, it suffices to show that the push-forward kernel $\Gamma^* \Phi (x,\cdot)  = \Phi(x,\Gamma^{-1}(\cdot))$ satisfies the minorization condition. It is easy to see that, given the assumptions on the noise, the kernel $\Phi(x,\cdot)$ has a density with respect to Lebesgue measure, so we simply need to show that the pushforward inherits the density property from $\Phi$. This can be understood as the controllability of the map $\Gamma$.  To achieve this, we use the following simple fact, which is lemma 5.4 in \cite{TMK15non}. 
\begin{lem}
\label{lem:twosteps}
Let $\Phi$ be a Markov transition kernel from $\reals^{n} \to \reals^{n} \times \reals^m$ with a Lebesgue density $p(x,y) = p(x,(y_1,y_2))$ and let $\Gamma : \reals^n \times \reals^m \to \reals^n$. 
%
%Let $\Phi(x,dy_1,dy_2)$ be a transition kernel with a density function $p(x,y_1,y_2)$ from space $\mathbb{R}^n$ to $\mathbb{R}^{n+m}$, and $\Gamma$ is a mapping from $\mathbb{R}^{n+m}$ to $\mathbb{R}^n$. 
Given a compact set $C$, suppose that there is a point $y^*=(y_1^*, y_2^*) \in \reals^n \times \reals^m$ and $\beta > 0 $ such that 
\begin{enumerate}
\item Reachable from all $x\in C$, that is the density function $p(x,y)>\beta$ for $y$ around $y^*$\;,
\item $\Gamma$ is controllable around $y^*$, that is $\Gamma$ is $C^1$ near $y^*$ and $\det (\CD_{y_1} \Gamma)|_{y^*}>0$\;. 
\end{enumerate}
Then there is a $\delta>0$ and a neighborhood $O_1$ of $\Gamma(y^*)$ such that for all $x\in C$ 
\[
\Gamma^* \Phi(x, \cdot )\geq \delta\lambda_{O_1}(\cdot)
\]
where $\lambda_{O_1}$ is the Lebesgue measure restricted to the set $O_1$. In other words, the minorization condition holds for the transition kernel $\Gamma^*\Phi$. 
\end{lem}

\subsection{Application to ensemble filters with adaptive inflation}
To apply Lemma \ref{lem:twosteps} to the proof of Theorem \ref{thm:geometric} we will use the variables
\begin{equ}
\label{eqn:Gammavaria}
\begin{gathered}
x = (U_{n-1} , V_{n-1}^{(1)},\dots,V_{n-1}^{(K)} )\\  y_1 = (U_n, \Vhat_n^{(1)},\dots,\Vhat_n^{(K)}) \\ y_2 = (Z_n, Z_n^{(1)},\dots,Z_n^{(K)} )
\end{gathered}
\end{equ}
The choice of the intermediate point $y^* = (y_1^*,y_2^*)$ can be quite delicate and can simplify the non-degeneracy condition considerably. In particular, we prove the following proposition in \cite{TMK15non}. 
\begin{prop}
\label{prop:control}
Let $\Phi$, $\Gamma$ be the Markov kernel and deterministic map used to define EnKF/ETKF/EAKF, given by the formulas in Appendix \ref{app:ergo_formulas}. If the system noise is non-degenerate (Assumption \ref{aspt:density}) then $\Phi,\Gamma$ satisfies both conditions of Lemma \ref{lem:twosteps} with the following choice of intermediate point $y^*$:
%
%Let \todo{DK: $\Psi$ Bad notation. Why do you need $\Psi$ anyway?} $\Psi$ be the forecast transition kernel, and  $\Gamma$ be the Kalman update map in the ensemble based filters  without AACI. Then with non degenerate system noise, Assumption \ref{aspt:density},  the following points are both reachable and controllable for each type of filter respectively, in the sense that the conditions of Lemma \ref{lem:twosteps} hold. 
\begin{enumerate}
\item EnKF: $y^*$ has all its components  being at the origin.
\item ETKF: $y^*$ has all its components  being at the origin.
\item EAKF: $y^*_2$ has all its components being at the origin, $U^*$ is at the origin, while $[\Vhat^{*(1)},\ldots, \Vhat^{*(K)}]$ is  the $d \times K$ matrix given by 
\begin{equ}
\begin{cases} \begin{bmatrix} \vec{1} & \Xi_0 & 0 \end{bmatrix}  \quad \text{if $d \leq K-1$} \\  \\ 
\begin{bmatrix} \vec{1} &  \Xi_0 \\ 0 &  0 \end{bmatrix}  \quad \text{if $d \geq K-1$}  
\end{cases}
\end{equ}
where $\vec{1}$ is the $d\times 1$ vector of $1$'s and $\Xi_0$ is the $r \times r$ matrix  with $r=\min\{K-1, d\}$:
\begin{equ}
\Xi_0 = \begin{bmatrix} 1 & 1 & \dots & 1 & -r \\
\vdots & \vdots & \vdots & \iddots & 0 \\
1 & 1 & -3 & \dots & 0\\ 
1 & -2 & 0 & \dots & 0 \\
-1 & 0 &  0 & \dots & 0 
\end{bmatrix}.
\end{equ} 
Note in this case, a version of EAKF is picked so $\Gamma$ is $C^1$ near $y^*$. 
\end{enumerate}
\end{prop}
With this propsition at hand, Theorem \ref{thm:geometric} becomes an immediate corollary.
\begin{proof}[Proof of Theorem \ref{thm:geometric}]
For each of EnKF-AI, ETKF-AI, EAKF-AI verification of the reachability condition in Lemma \ref{lem:twosteps} is identical to that of EnKF, ETKF, EAKF which is trivial and is done in \cite[Proof of theorems 5.5, 5.6, 5.8]{TMK15non} respectively. Hence it suffices to check the non-degeneracy of $\Gamma$ in each of the three cases. For EnKF-AI and ETKF-AI we pick $y^* = (y_1^*,y_2^*)$ to be the origin. Hence for thresholds $\Me,\Mc > 0$, there exist a neighborhood of the origin such that the map $\Gamma$ coincides with that of EnKF, ETKF respectively, since the inflation mechanism is not triggered within this neighborhood. Hence the non-degeneracy follows immediately from Proposition \ref{prop:control}. For EAKF-AI, as detailed above, the intermediate point $y^*$ is not chosen at the origin but rather at the point specified in Proposition \ref{prop:control}. Nevertheless, since this point is fixed, one can pick thresholds $\Me,\Mc$ sufficiently large so that for some neighborhod of $y^*$ the inflation mechanism is not triggered and hence $\Gamma$ agrees with that of EAKF. Hence the non-degeneracy follows from Proposition \ref{prop:control}.   
\end{proof}
%
%
%The reachability part of Proposition \ref{prop:control} is quite simple see, since due to Assumption \ref{aspt:density}, for any two points $x, y$ inside two fixed bounded set, the transition density $p(x,y)$ is bounded from below. The controllability is rather easy to verify for EnKF and ETKF, since with components at the origin, one will find that $\Gamma$ is a projection onto $y_1$. But for EAKF, the computation is much more complicated, much due to its complicated formulation. A perturbation analysis for matrix decomposition was developed in \cite{TMK15non} for this purpose. 
%
%Now back to the filters with AACI. As a matter fact, the claims of Proposition \ref{prop:control} hold with the identical proof as long as the threshold $M$ is sufficiently large. To see this, we simply need to recall that the AACI takes effect only when $\Eo$ or $\Eo(\Cu+1)$ is beyond $M$, else the operator $\Gamma$ is identical to the one of without AACI. Then because the $y^*$s introduced by Proposition \ref{prop:control} are all fixed points, around them $\Eo$ and $\Eo(\Cu+1)$ will be of bounded value. So when $M$ is large enough for $y^*$ to be in the acceptance set, Proposition \ref{prop:control} naturally extends to the filter ensemble with AACI. This leads to the proof of Theorem \ref{thm:geometric}. 
%

\section{Thresholds from elementary benchmarks}
\label{sec:choice}\label{s:choice}
Finally we discuss the choice of the cutoff thresholds $\Me$ and $\Mc$ in the formulation \eqref{e:lambda} of $\cut$. While Theorems \ref{thm:adEnKF} guarantees filter stability for any $\Me, \Mc > 0$, certain values produce better filter performance. Intuitively, these thresholds should be chosen to differentiate malfunctioning forecast ensembles from properly working ones.

\subsection{An elementary benchmark}
Here we devise an elementary benchmark of accuracy that should be surpassed by any properly functioning filter. At each observation step, the benchmark estimate is given by the stationary state $U_n$ conditioned on the observation $Z_n$. This can be computed using Bayes' formula and only requires access to the invariant measure $\pi$ of the model. For instance we can compute the conditional expectation via
\begin{equ}\label{e:bayes_integral}
\mathbb{E}(U_n|Z_n)=\frac{ \int x \pi(x)\exp(-\frac{1}{2}|Z_n-Hx|^2) dx }{\int \pi(x) \exp(-\frac{1}{2}|Z_n-Hx|^2) dx}\;.
\end{equ}
The benchmark for accuracy is given by the mean-square error of the above estimator:
 \[
\text{Error}_{A}:=\mathbb{E}|U_n-\mathbb{E}(U_n|Z_n)|^2=\mathbb{E}|U_n|^2-\mathbb{E}|\mathbb{E}(U_n|Z_n)|^2\;.
\]
A properly functioning forecast ensemble should be expected to perform better than this estimator. For instance, one should expect a properly functioning filter to satisfy
\[
\frac{1}{K}\sum_{k=1}^K\mathbb{E}\left|\Vhat^{(k)}_n-U_n\right|^2\leq \text{Error}_A\;.
\]
Consequently, one should expect that 
\[
\mathbb{E}\Eo^2=\frac{1}{K}\sum_{k=1}^K\mathbb{E}|H(\Vhat^{(k)}_n-U_n)+\xi_n+\xi_n^{(k)}|^2\leq \sigma_{\Theta}^2:=\|H\|^2\text{Error}_A+2d\;.
\]
Moreover using the inequalities $\|a\otimes b\|\leq |a||b|$, $|x||y|\leq \frac{1}{2}(x^2+y^2)$ and Lemma \ref{lem:mean} we expect
\begin{align*}
\mathbb{E}\Cu&\leq \frac{1}{K-1}\sum_{k=1}^K\mathbb{E}\left|\Xhat^{(k)}_n-\overline{\Xhat}_n\right|\left|\Yhat^{(k)}_n-\overline{\widehat{Y}}_n\right|
\leq \frac{1}{2K-2}\sum_{k=1}^K\mathbb{E}|\Vhat^{(k)}_n-\overline{\Vhat}_n|^2\\
&\leq \frac{1}{2K-2}\sum_{k=1}^K\mathbb{E}|\Vhat^{(k)}_n-U_n|^2\leq M_{\Xi}:=\frac{K}{2K-2}\text{Error}_A.  
\end{align*}
In summary, we expect that a properly functioning filter should satisfy \todo{M notation here?}
\begin{equation}
\label{eqnbench:filteronce}
\mathbb{E}\Eo^2\leq \sigma_{\Theta}^2:= \|H\|^2 \text{Error}_A+2d,\quad 
\mathbb{E}\Cu\leq M_{\Xi}:= \frac{K}{2K-2}\text{Error}_A\;.
\end{equation}
Since the integrals appearing in the Bayesian formula \eqref{e:bayes_integral} are often difficult to compute, it is more useful to have an estimator that only relies on low order statistics of $\pi$. For this purpose it is natural to introduce a Gaussian approximation of $\pi$, replacing $\pi$ with $\pitilde = \mathcal{N}(\mathbb{E} U_n, \text{cov} (U_n))$ where the mean and covariance are the true statistics computed from $\pi$. The condtional distribution of $\pitilde$ given the observation $Z_n$ is now a Gaussian measure and can be computed exactly. In particular we have that 
\[
\widetilde{\mathbb{E}}(U_n|Z_n)
=\mathbb{E}U_n-\text{cov}(U_n)H^T(I+H^T\text{cov}(U_n) H)^{-1}(H\mathbb{E}U_n-Z_n). 
\]
The  covariance of the error $r_K=U_n-\widetilde{\mathbb{E}}(U_n|Z_n)$  is given by the Kalman prior-posterior covariance relation
\begin{equation}
\label{eqnbench:covKal}
\text{cov}(r_K)=\text{cov}(U_n)-\text{cov}(U_n)H^T(I+H\text{cov}(U_n)H^T)^{-1}H\text{cov}(U_n),
\end{equation}
which can be explicit computed. The mean square error is given by 
\[
\text{Error}_{A}=\mathbb{E}|r_K|^2=\text{tr}(\text{cov}(r_K)).
\]
In the case of a linear model, this Gaussian approximation is of course exact. 

Given the above discussion, an obvious option for the thresholds is 
\begin{equation}
\label{threshold:aggressive}
\Me=\sigma_{\Theta},\quad \Mc=M_{\Xi}.
\end{equation}
we call this \emph{aggressive thresholding}. In this situation the adaptive inflation is triggered as soon as the filter statistics $\Eo, \Cu$ exceed the mean values computed from the crude estimators. A less aggressive strategy is to obtain $\Me,\Mc$ by hypothesis testing, which we do not pursue here.

\begin{rem}
To avoid wasting resources on computing these benchmarks, one could instead use the trivial estimator $U_n = 0$. In this case, one can compute
\[
\mathbb{E}|U_n-\mathbf{0}|^2=\mathbb{E}|U_n|^2\leq \beta_h^{-1}K_h\;.
\]
A properly functioning EnKF should have a smaller mean square error in its forecast. And as a consequence, we can replace term $\text{Error}_A$ in  the bounds above with $\beta_h^{-1}K_h$, which leads to  
\[
\mathbb{E}\Eo^2\leq \|H\|^2\beta_h^{-1}K_h+2d,\quad \mathbb{E}\Cu\leq \frac{K}{2K-2}\beta_h^{-1}K_h. 
\]
\end{rem}

\section{Numerical results}
\label{sec:num}\label{s:numerics}
In this section, we will numerically validate our results through the following five mode Lorenz 96 model:
\begin{equation}
\label{sys:lz96}
\frac{d}{dt}x_i=x_{i-1}(x_{i+1}-x_{i-2})-x_i+F,\quad i=1,\ldots,5.
\end{equation}
Here the indices are interpreted in modulo $5$ sense. The truth will be given by $U_n = x (n h)$ for some observation time step $h$. Previous studies from \cite{MAG05, GM13, MH12} have shown that different choice of forcing $F$ produces very different strength of turbulence.  Here we will pick forcing $F=4, 8$ and $16$,  and see that these choices create very different filtering results.

%[This paragraph is not for publication ]In Figure \label{fig:decor} below, we plot the time correlation of system \eqref{sys:lz96} using a simulation of length $T=10000$, the time correlation is given by 
%\[
%\text{Corr}_t=\frac{T\int^{T-t}_0 \langle X_s, X_{s+t}\rangle ds }{(T-t)\int^{T}_0 | X_s|^2 ds}.
%\]
%\begin{figure}[!h]
%\hspace{-1.5cm}\scalebox{0.6}{\includegraphics{decor.eps}}
%\caption{Time correlation of Lorenz 96 model with different forcing. This figure is not for publication
%}
%\label{fig:decor}
%\end{figure}
\par
Throughout the section we assume that only one component $x_1$ is observed, hence we have the observation matrix $H = [1,0,0,0,0]$. The observational noise is chosen to be distributed as $\mathcal{N}(0,0.01)$. 
%To bring the numerical results in line with the theoretical results, we will perform the change of coordinates detailed in Section \ref{s}, which amounts to the noise being distributed as $\mathcal{N}(0,0.01)$ and the observation matrix $H= (1,0)$ 
%
%
%This choice of observation can be transform to the standard form with $\mathcal{N}(0,1)$ observation noise with
%\[
%H=\begin{pmatrix}
%10 &0 &0 &0 &0
%\end{pmatrix}. 
%\]
We apply a selection of ensemble assimilation methods to this signal-observation system with $6$ ensemble members, where the initial positions are drawn from a  Gaussian approximation of the equilibrium measure. For the sake of simplicity, we integrate \eqref{sys:lz96} using an explicit Euler scheme with time step $\Delta=10^{-4}$. For most ensemble based filtering methods, the explicit Euler scheme is bad choice of integrator for \eqref{sys:lz96}, since the stiffness of the equation will lead to numerical blow-up for sufficiently large initial conditions. As we shall see, this leads to prevalence of catastrophic filter divergence in most ensemble based filtering methods. The adaptive inflation method on the other hand avoids this issue and exhibits high filtering skill. 
\par
Here we compare the performance of four ensemble based filters: EnKF, EnKF-AI, EnKF-CI and EnKF-CAI, which have been defined in Sections \ref{s:enkf} and \ref{s:ai}. The constant inflation is always taken to be additive.

%The first two filters have been defined in Sections \ref{s:enkf_algo} and \ref{s:aaci_algo} respectively. CCI is the standard method of additive covariance inflation, where by in the standard EnKF algorithm one replaces $\Chat_n$ with $\Chat_n + \rho I$ for a constant scalar $\rho > 0$. C+AACI is a combination of the CCI and AACI, in that one replaces $\Chat_n$ with $\Chat_n + \rho I + \lambda_n I$ where $\rho > 0$ is a constant scalar and $\lambda_n > 0$ is the adaptive scalar used in the AACI method.  
%
%
%The first one is the original EnKF, we will use EnKF to denote it in the following. The second one is applying AACI to the original EnKF, so the covariance $\Chat_n+\lambda I$  with $\lambda$ given by \eqref{eqn:lambda}  is used for analysis step, we use AACI to denote this filter. The third one applies constant additive inflation to the original EnKF, so the covariance $\Chat_n+\rho I$ with $\rho=0.1$ is used for the analysis step, we use CCI to denote this filter. The last one combining CCI and AACI, so we use $\Chat_n+\rho I+\lambda I$ to update the analysis step, and we use C+AACI to denote this algorithms. 
\par
Following the methodology introduced in Section \ref{sec:choice}, we construct the equilibrium measure of the Lorenz 96 system using data from a long time ($T=10^4$) simulation, which leads to the performance benchmarks computed by \eqref{eqnbench:covKal} and \eqref{eqnbench:filteronce}. The thresholds $\Me$ and $\Mc$ are obtained through the aggressive strategy \eqref{threshold:aggressive}, using the Gaussian approximation of the stationary measure $\pi$. In order for the comparison to be unbiased, we use the same realizations of observation noise $\xi_n$ and perturbations $\xi_n^{(k)}$ in all four filters. The statistics are collected for $N=100$ independent trials, where each runs a total time length of $T=100$. To eliminate transients, we will only record data from the period $50 \leq T \leq 100$. 
\par
The performances of the four filters are measured and compared from three perspectives:
\begin{itemize}
\item The most important question here is how often does catastrophic filter divergence appear in standard (non-adaptive) filters, like EnKF and EnKF-CI, and does the adaptive inflation in EnKF-AI and EnKF-CAI prevent it. Catastrophic filter divergence can be identified by checking  whether the ensemble at the final time $T=100$ takes value ``NaN", representing machine infinity, in any of its components.
\item The overall accuracy of the filters is compared using the average of the root mean square error (RMSE) over $N=100$ trials among the second half of the time interval, where
\[
\text{RMSE}=\sqrt{\frac{2}{T}\sum_{n=T/2h}^{T/h}|\Vbar_n-U_n|^2}. 
\]
Another statistics that is useful in judging the accuracy is the averaged pattern correlation:
\[
\text{Cor}=\frac{2}{T}\sum_{n=T/2}^T\frac{\langle\Vbar_n-\overline{U}, U_n-\overline{U}\rangle}{|\Vbar_n-\overline{U}||U_n-\overline{U}|}\;,
\]
where $\Ubar$ denotes the climatological average of the true model (the mean of the invariant measure) and $\Vbar_n$ as usual denotes the average of the filter ensemble. We will also plot the of the posterior error, $|\overline{V}_n-U_n|$, for one realization of a  typical trajectory for the four filters.
%
%\item When catastrophic filter divergence takes place among the $100$ trials for EnKF or CCI, both RMSE and pattern correlation will be corrupted as ``NaN". This obscures the information of the trials that did not exhibit catastrophic filter divergence, which may be very few. 
%%which describes the performance of the filters when they are functioning properly.  
%To avoid this issue, in a different set of tests, we will repeat the test until there are $100$ independent trials without catastrophic filter divergence. We will apply the AACI procedure to these trials and compare performances in this regime. We only focus on instances where the chance of catastrophic filter divergence is less than $90\%$, since the high chance of explosion makes it practically useless. 
\end{itemize}
%\par
%We will also study the change of filter performances when the constant inflation strength $\rho$ varies from $0.1$. This will reveal the 
%benefit of AACI procedure from a different level. 
%\par
%As the last step, we plot the histogram of $\Eo$ and $\Cu$ collected from the analysis steps of AACI. These data provide us a deeper review of the assumptions and thresholds we made in Section \ref{sec:choice}. 

 The following experiments have also been carried out with both ESRF methods. All these methods have performance very similar to EnKF based method, as well as the effect of adaptive inflation over them. We do not present these results for brevity. 

\subsection{Accuracy for different turbulent regimes}\label{s:num_turb}

In this section we compare the performance of the four filters across the three turbulent regimes. In each experiment we use observation interval $h = 0.05$ and (when required) constant additive inflation strength $\rho = 0.1$.

\subsubsection{Weak turbulence}

Here we compare the performance of the filters in a weakly turbulent regime, by setting $F=4$. The performances of the four filters are presented in Table \ref{tab:F4}, where we compare frequency of catastrophic filter divergence, RMSE and pattern correlation. In Figure \ref{fig:F4}, we compare the posterior error for the four filters for one typical trajectory. 
\par
We found that, among $100$ trials, none of the filters experience catastrophic filter divergence. The adaptive inflation in EnKF-AI  has  been triggered in $30$ trials, and on average $1.96$ times in each trial. In EnKF-CAI, the adaptive inflation has been turned on in $9$ trials, and only once in each trial. Hence the typical realizations of EnKF and EnKF-AI are identical and similarly for EnKF-CI and EnKF-CAI, since the adaptive inflation is seldom triggered. 
\par
In terms of accuracy, all four filters perform quite well. For each filter the RMSE is significantly lower than the benchmark RMSE, which is $3.64$ and the pattern correlation is very high. Unsurprisingly, the  filters with constant covariance inflation (EnKF-CI and EnKF-CAI) have more skill than EnKF, with lower RMSE and higher pattern correlation. Somewhat more surprising is that EnKF-AI also performs better than EnKF, both in RMSE and pattern correlation, even though the adaptive covariance inflation has only been switched on in $30$ of the $100$ trials. This indicates that the performance of EnKF must be quite bad in those $30$ and the adaptive inflation is needed to ensure the accuracy. Of course, in this regime one could also achieve this with non-adaptive constant inflation.  
\begin{table}[h]
\begin{center}
\begin{tabular}{|c| c c c c|}
\hline  
Filter &EnKF &EnKF-AI &EnKF-CI &EnKF-CAI\\
\hline
Cata. Div. &0\%  &0\% &0\% &0\%\\ 
RMSE  &0.89 &0.54 &0.22 &0.22\\
Pattern Cor. &0.91 &0.96 &0.98 &0.98\\
\hline
\end{tabular}
\caption{In the weak turbulence regime ($F=4$), for the four algorithms we compare frequency of catastrophic filter divergence (Cata. Div.), RMSE, and pattern correlation (Pattern Cor.) The benchmark RMSE is $3.25$. }
\label{tab:F4}
\end{center}
\end{table}

\begin{figure}[h]
\hspace{-2 cm}\scalebox{0.50}{\includegraphics{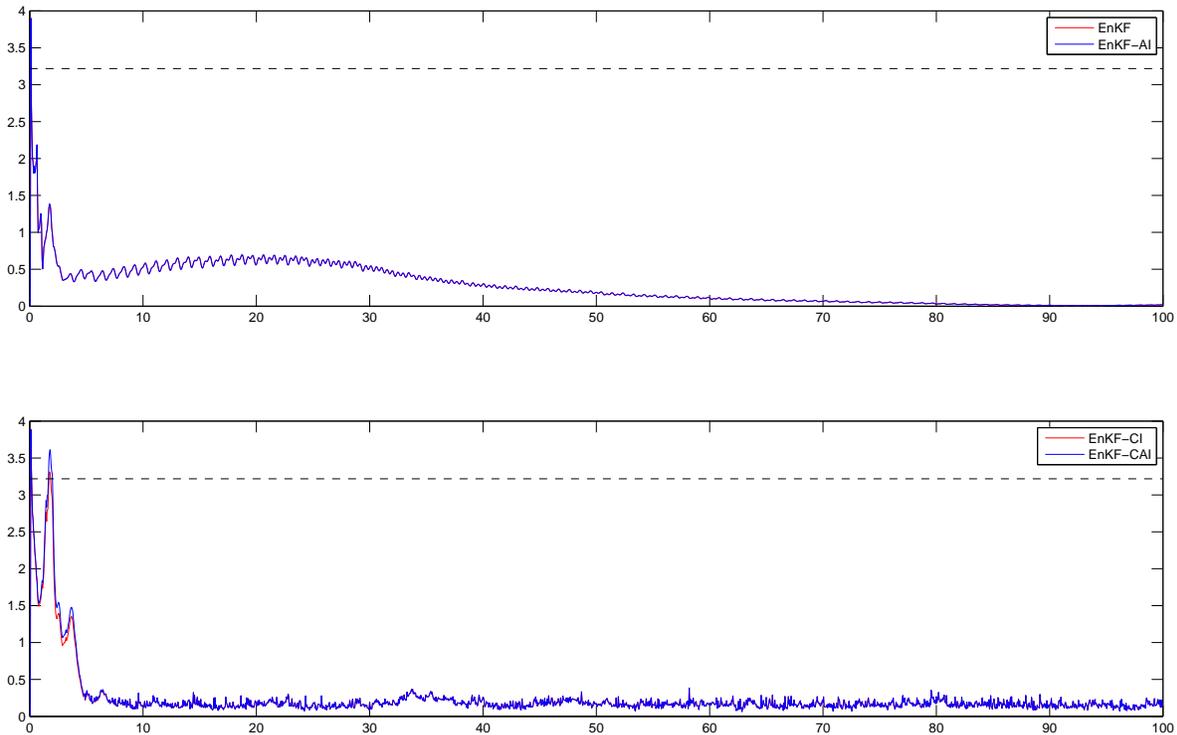}}
\caption{We compare the posterior error for one realization of the four filters in the weakly turbulent ($F=4$) regime. Since the adaptive inflation is not triggered, trajectories collapse into two groups (those with constant inflation and those without). The benchmark RMSE is included (black dash) to emphasize the accuracy of the filters. }
\label{fig:F=4}\label{fig:F4}
\end{figure}

%
%\begin{center}
%\begin{tabular}{|c| c c c c|}
%\hline  
%Filter &EnKF &AACI &CCI &C+AACI\\
%\hline
%Cata. Div. &0\%  &0\% &0\% &0\%\\ 
%RMSE  &0.84 &0.49 &0.22 &0.22\\
%Pattern Cor. &0.91 &0.96 &0.98 &0.98\\
%\hline
%\end{tabular}
%\end{center}

%Also, the RMSE at the observed direction $x_1$ is $0.25$, which is much smaller than the benchmark $\sqrt{0.78}=0.88$. 

We now briefly describe how the thresholds and benchmarks were calculated. The equilibrium mean on each mode is $\bar{x}_i\approx 1.22$ and the variance is $\text{var}(x_i)\approx 3.38$. Using  Kalman one time assimilation, the posterior variance of $x_1$ will be reduced to $0.01$, while the benchmark RMSE will be $3.25$.  From these statistics, we compute $\Me=\sigma_{\Theta}=32.5,\quad \Mc=M_{\Xi}=6.2$.

\subsubsection{Moderate turbulence}
Here we compare the performance of the filters in a moderately turbulent regime, by setting $F=8$. The performances of the four filters are presented in Table \ref{tab:F8}, where we compare frequency of catastrophic filter divergence, RMSE and pattern correlation. In Figure \ref{fig:F8}, we compare the posterior error for the four filters for one typical trajectory. 
\par
Among $100$ trials,  catastrophic filter divergence takes place $12$ times for EnKF, but never takes place for the other three. The adaptive inflation in EnKF-AI has been turned on in $96$ trials, while on average $17.5$ times in each of these trials. As for EnKF-CAI, the adaptive inflation has been turned on in $20$ trials, while on average $6.05$ times each of these trials. 
\par
In terms of accuracy, EnKF will not be analyzed due to the frequent occurrence of catastrophic filter divergence. The EnKF-AI avoids the issue of catastrophic filter divergence, however it has very little skill, since the RMSE is higher than the benchmark RMSE $7.02$ and the pattern correlation is quite low. However, the filters with constant inflation (EnKF-CI and EnKF-CAI), behave almost identically, display significantly better skill, beating the benchmark RMSE and having high pattern correlation.  

\begin{table}[h!]
\begin{center}
\begin{tabular}{|c| c c c c|}
\hline  
Filter &EnKF &EnKF-AI &EnKF-CI &EnKF-CAI\\
\hline
Cata. Div. &12\% &0\% &0\% &0\%\\
RMSE  &NaN &8.6 &3.61 &3.57\\
Pattern Cor. &NaN &0.55 &0.89 &0.89\\
\hline
\end{tabular}
\caption{In the moderate turbulence regime ($F=8$), for the four algorithms we compare frequency of catastrophic filter divergence (Cata. Div.), RMSE, and pattern correlation (Pattern Cor.) The benchmark RMSE is $7.02$. }
\label{tab:F8}
\end{center}
\end{table}

\begin{figure}[h!]
\hspace{-2cm}\scalebox{0.5}{\includegraphics{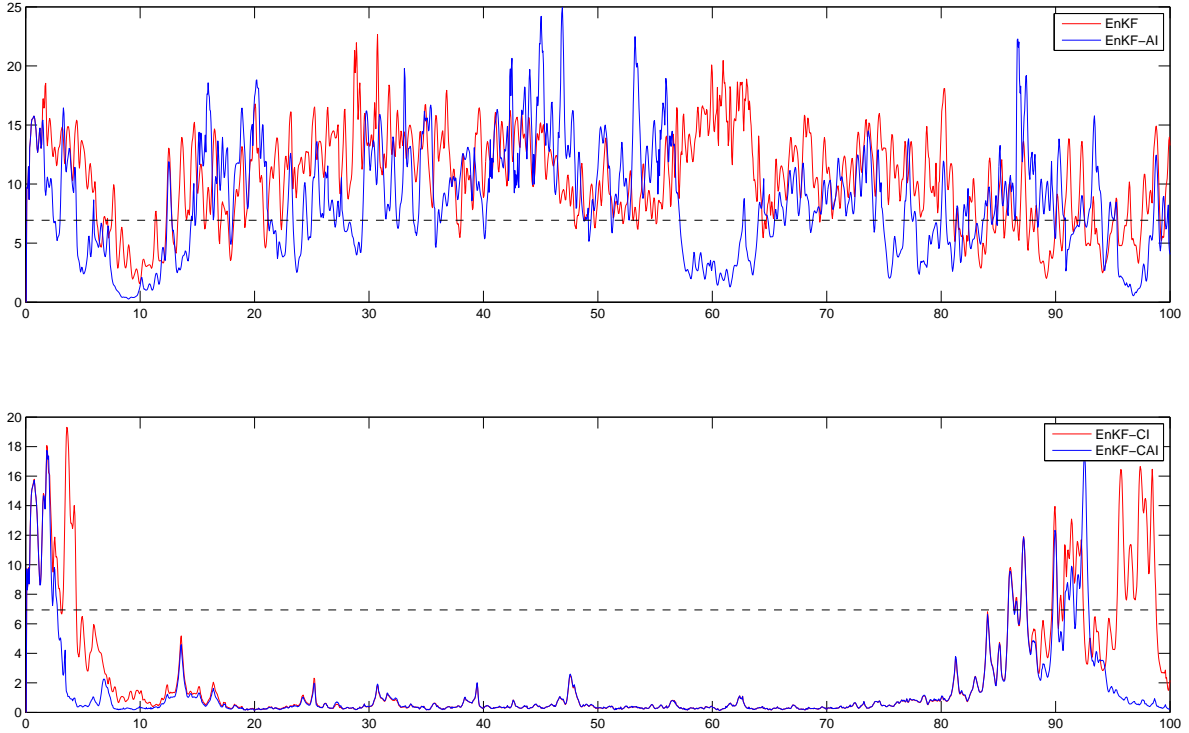}}
\caption{We compare the posterior error for one realization of the four filters in the moderately turbulent ($F=8$) regime. The EnKF and EnKF-AI are not performing well, as they are frequently exceeding the benchmark (black dashed). Adding constant covariance inflation significantly improves the performance.
}
\label{fig:F=8}\label{fig:F8}
\end{figure}

%
%To analyze the accuracy in the presence of (infrequent) catastrophic filter divergence, we rerun EnKF until there are 100 trails without explosions. The corresponding RMSE is  $9.18$ and pattern correlation is $0.45$. Applying AACI to these trials, the performance is improved, with RMSE being $8.7$ and pattern correlation $0.55$. \todo{DK: Only EnKF and AACI?}
%\par
%In this turbulent regime, EnKF has little to no skill. Its RMSE is worse than the benchmark $7.61$ even if catastrophic filter divergences are ruled out. The application of AACI improves the skill slightly, but still not enough to reach the benchmark. On the other hand, constant inflation improves the performance by a lot. Adding AACI to CCI improves the skill slightly. 
%\par

We now briefly describe how the thresholds and benchmarks were calculated. The equilibrium mean on each mode is $\bar{x}_i\approx 2.28$ and the variance is $\text{var}(x_i)\approx 12.6$. Using  Kalman one time assimilation, the posterior variance of $x_1$ will be reduced to $0.01$, while the benchmark RMSE will be $7.02$.  From these statistics, we compute $\Me=\sigma_{\Theta}=69.56$ and $\Mc=M_{\Xi}=28.8$. 
%
%The equilibrium mean on each mode is $\bar{x}_i\approx 2.28$ while the variance is $\text{var}(x_i)\approx 12.6$. Using  Kalman one time assimilation, the posterior variance of $x_1$ will be reduced to $0.01$, while the benchmark RMSE will be $6.9$.  From these statistics, we can compute the variation benchmark for $\Eo$ and $\Cu$: 
%\[
%\Me=\sigma_{\Eo}=69.56,\quad \Mc=M_{\Cu}=14.37.
%\]

 \subsubsection{Strong turbulence}
 Here we compare the performance of the filters in a highly turbulent regime, by setting $F=16$. The performances of the four filters are presented in Table \ref{tab:F16}, where we compare frequency of catastrophic filter divergence, RMSE and pattern correlation. In Figure \ref{fig:F16}, we compare the posterior error for the four filters for one typical trajectory. 
\par
 In this highly turbulent regime, both non-adaptive filters, EnKF and EnKF-CI display frequent catastrophic filter divergence, making them practically useless.  Among $100$ trials, catastrophic filter divergence occurs in every trial of EnKF and in $18$ trials of EnKF-CI. Both EnKF-AI and EnKF-CAI experience no catastrophic filter divergence. The adaptive inflation has been turned on in all trials of EnKF-AI, occurring $97.53$ times on average per trial. The adaptive inflation has been turned on for EnKF-CAI in $80$ trials, occurring 18.7 times on average per trial.
\par
In terms of accuracy, the EnKF-AI exhibits very little skill, with RMSE significantly higher than the benchmark RMSE $12.93$ and with low pattern correlation. The EnKF-CAI however does slightly beat the benchmark RMSE and displays moderate pattern correlation. In Section \ref{s:num_tuning} below, we show that the performance of EnKF-CAI is sharply improved by a simple tuning of parameters.

\begin{table}[h!]
\begin{center}
\begin{tabular}{|c| c c c c|}
\hline  
Filter &EnKF &EnKF-AI &EnKF-CI &EnKF-CAI \\
\hline
Cata. Div. &100\% &0\% &18\% &0\%\\
RMSE  &NaN &24.48 &NaN &11.91\\
Pattern Cor. &NaN &0.23 &NaN &0.69\\
\hline
\end{tabular}
\caption{In the highly turbulence regime ($F=16$), for the four algorithms we compare frequency of catastrophic filter divergence (Cata. Div.), RMSE, and pattern correlation (Pattern Cor.) The benchmark RMSE is $12.93$. }
\label{tab:F16}
\end{center}
\end{table}

 \begin{figure}[h!]
\hspace{-2cm}\scalebox{0.5}{\includegraphics{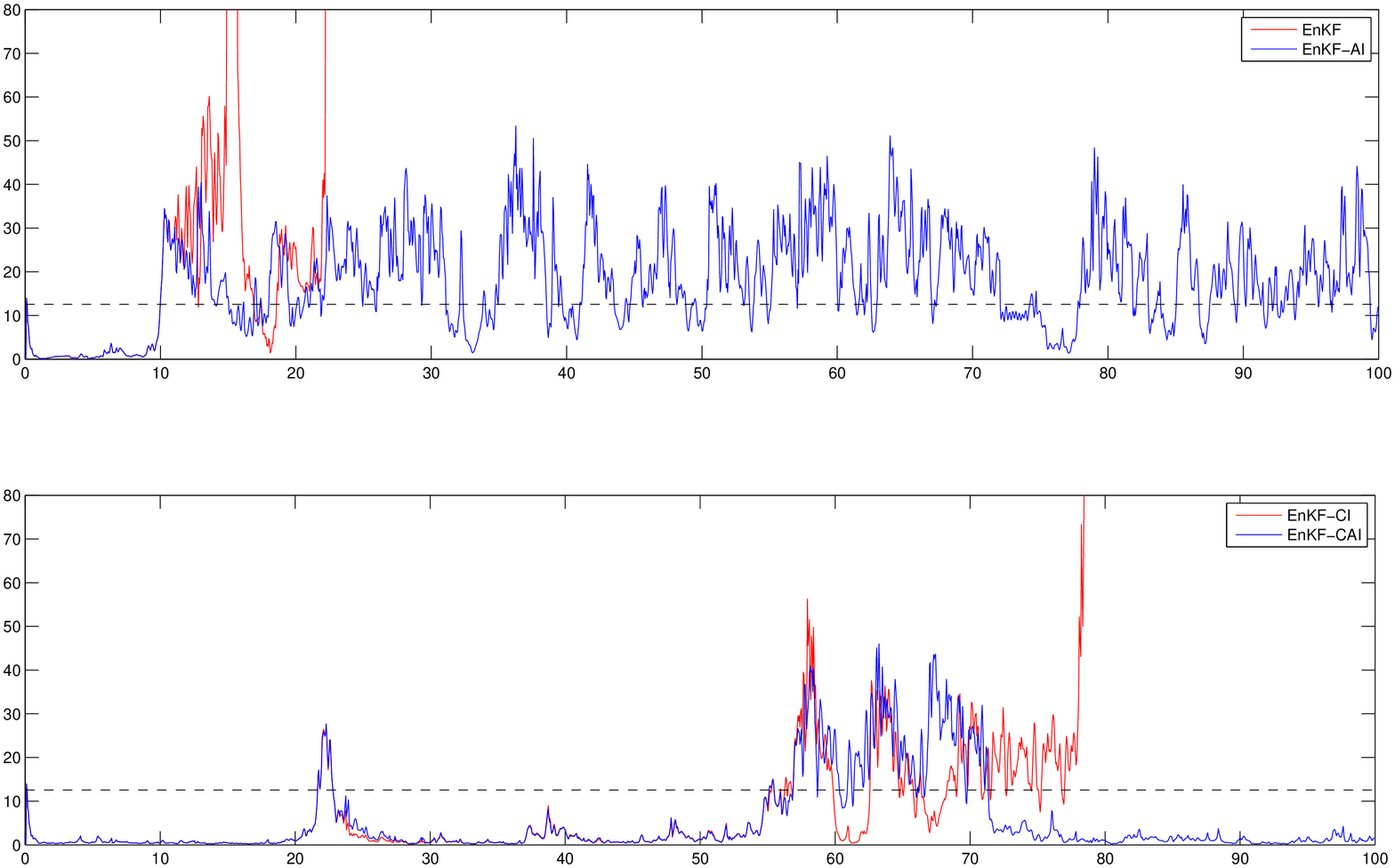}}
\caption{We compare the posterior error for one realization of the four filters in the strongly turbulent ($F=16$) regime. The filter is very unstable, EnKF (red, upper panel) explodes to machine infinity, and EnKF-CI (red, lower panel) does the same with significant probability. The adaptive inflation mechanism, when applied to these two filters, is triggered frequently, preventing the ensemble from exploding. The performance of EnKF-AI (blue, upper panel)  is worse than the benchmark (black dash). The performance of EnKF-CAI (blue, lower panel), is much better and outperforms the benchmark. 
\label{fig:F=16}\label{fig:F16}}
\end{figure}
%AACI mechanism has been turned on in all trials of AACI, with $97.53$ times on average. AACI mechanisms has been turned on for C+AACI in $80$ trials, with average 18.7 times each trial. The RMSE of the four filters are:
%\[
%\text{EnKF:NaN} (-),\quad \text{AACI: }24.48(-),\quad \text{CCI: NaN}(13.98) \quad \text{C+AACI: }11.67 (11.70).
%\]
%The pattern correlation are
%\[
%\text{EnKF:NaN} (-),\quad \text{AACI: }0.23(-),\quad \text{CCI: NaN}(0.76) \quad \text{C+AACI: }0.75 (0.74).
%\]
%We try to remove the catastrophic filter divergence effect on CCI by repeating the tests until there are 100 trials without explosion. The corresponding RMSE is $13.43$ and pattern correlation is $0.69$. Applying AACI to these trials improves these statistics to $11.39$ and $0.70$. Apparently, in this strongly turbulent regime, only C+AACI has an accuracy surpasses the benchmark, and the application of AACI is crucial to the stability of the filters.  
We now briefly describe how the thresholds and benchmarks were calculated. The equilibrium mean on each mode is $\bar{x}_i\approx 3.1$ and the variance is $\text{var}(x_i)\approx 40.6$. Using  Kalman one time assimilation, the posterior variance of $x_1$ will be reduced to $0.01$, while the benchmark RMSE will be $12.93$.  From these statistics, we compute $\Me=\sigma_{\Theta}=127.6$ and $\Mc=M_{\Xi}=81.4$. 

\newpage
\subsection{Dependence on  constant inflation strength and observation time}\label{s:num_tuning}
Based on previous discussions, the performance of EnKF-CI and EnKF-CAI are satisfactory in the weakly and moderately turbulent regimes, but the strongly turbulent regime poses a significant challenge. In this section, we study how constant covariance inflation strength $\rho$ and observation time $h$ affect the performance of EnKF-CI and EnKF-CAI in the $F=16$ regime. 
\par
In Table \ref{tab:rho} we compare the performance of EnKF-CAI for different choices of constant inflation strength $\rho$, including RMSE and pattern correlation. We also include the frequency of catastrophic filter divergence in EnKF-CI to emphasize the impact of adaptive inflation. 
 \par
As constant inflation strength is decreased, the frequency of catastrophic filter divergence increases. This agrees with the notion that removing inflation simply results in EnKF, which in Table \ref{tab:F16} was shown to experience catastophic filter divergence in $100\%$ of trials.  
\par
In terms of accuracy, the EnKF-CAI can be optimized by choosing a moderately small constant inflation strength. In particular, when $\rho$ is in the $0.01-0.02$ range the RMSE is significantly smaller than the benchmark RMSE $12.93$ and the pattern correlation is relatively high. Surprisingly, the optimal choice of $\rho$ for EnKF-CAI is very sub-optimal for EnKF-CI, with around half of all trials experiencing catastrophic filter divergence. This emphasizes the importance of creating the right balance between adaptive inflation and constant inflation.   

\begin{table}[h!]
\begin{center}
\begin{tabular}{|c ||c|c|c|c|c|c|c|c|}
\hline
$\rho$ &1 &0.5 &0.2 &0.1 &0.05 &0.02 &0.01 &0.005\\
\hline
EnKF-CI Cata. Div. &3\% &8\% &18\% &18\% &28\% &42\% &57\% &75\%\\
\hline	
EnKF-CAI RMSE &13.05 &13.62 &13.43 &11.91 &8.82 &8.51 &9.3 &10.51\\
EnKF-CAI Cor. &0.64&0.65&0.66&0.69&0.70&0.70&0.75&0.70 \\
\hline
\end{tabular}
\end{center}
\caption{In the $F=16$ regime we display for different choice of constant inflation strength $\rho$, the frequency of catastrophic filter divergence for EnKF-CI (EnKF-CI Cata. Div.), the RMSE for EnKF-CAI and the pattern correlation for EnKF-CAI (EnKF-CAI Cor). Each statistic is generated from $100$ independent trials. The benchmark RMSE is $12.93$.} 
\label{tab:rho}
\end{table}

%We can further remove the effect of catastrophic filter divergence from CCI by collecting data from 100 trials without explosion. 

%We also apply AACI to these trials and compare the statistics.
%\par
%\begin{center}
%\begin{tabular}{|c ||c|c|c|c|c|c|c|c|}
%\hline
%$\rho$ &1 &0.5 &0.2 &0.1 &0.05 &0.02 &0.01 &0.005\\
%\hline
%CCI  RMSE &14.85 &14.47 &13.72 &13.43 &10.83 &9.26 &6.22 &6.41\\
%CCI  Cor. &0.64&0.64&0.69&0.69&0.75&0.81&0.86&0.88\\
%\hline
%\hline
%C+AACI RMSE&13.32 &13.72 &13.48 &11.39 &9.70 &7.94 &5.29 &6.80\\
%\hline
%C+AACI Cor. &0.65 &0.63 &0.65 &0.70 &0.73 &0.77 &0.79 &0.76\\
%\hline
%\end{tabular}
%\end{center}
%\par
%By combining the two tables, we can find a remarkable pattern. A smaller constant inflation $\rho$ around $0.01$ to $0.02$ improves the overall performance, but also significantly increase the chance of catastrophic filter divergence. This is not surprising, as big covariance inflation sacrifices the information from prior iterations, while small inflation approaches the scenario where there is no inflation and frequent catastrophic filter divergence. The application of AACI is crucial here to resolve this dilemma, as we can pick a small inflation strength that has the optimal performances, and do not need to worry about catastrophic filter divergence.

In Table \ref{tab:obs} we display the results of a similar experiment, now varying the observation interval $h$ keeping a fixed inflation strength $\rho=0.1$. Again we display frequency of catastrophic filter divergence for EnKF-CI and the and pattern correlation RMSE. Each statistic is generated from $100$ independent trials.  
\par
We found that the frequnecy of catastophic filter divergence is not monotonic in $h$, but rather displays a peak at $h = 0.1$. This agrees with the findings of \cite{GM13}. We also find that the accuracy of EnKF-CAI is not monotonic in $h$ but rather, there is an optimal choice around $h=0.1$. As with the inflation strength experiment, the optimal parameter choice for the accuracy of EnKF-CAI occurs when EnKF-CI is exhibiting frequent catastrophic filter divergence. The fact that filter performance decreased with decreasing $h$ can be attributed to a violation of filter observability conditions when the time-step becomes too small \cite{MH12}.

 \begin{table}[h!]
\begin{center}
\begin{tabular}{|c ||c|c|c|c|c|c|c|c|}
\hline
$h$ &0.01 &0.02 &0.05 &0.1 &0.2 &0.5\\
\hline
EnKF-CI Cata. Div. &0\% &1\% &18\% &25\% &5\% &0\% \\
\hline	
EnKF-CAI RMSE &25.75 &20.71 &11.91 &6.43 &14.09 &14.80\\
EnKF-CAI Cor. &0.31&0.37&0.69&0.64&0.50&0.36 \\
\hline
\end{tabular}
\caption{In the $F=16$ regime we display for different choice of observation interval $h$, the frequency of catastrophic filter divergence for EnKF-CI (EnKF-CI Cata. Div.), the RMSE for EnKF-CAI and the pattern correlation for EnKF-CAI (EnKF-CAI Cor). Each statistic is generated from $100$ independent trials. The benchmark RMSE is $12.93$.}
\label{tab:obs}
\end{center}
\end{table}
%
%Although shorter observation time brings in more frequent information, the filter performance here does not monotonically  improve. This can be interpreted as a delicate effect of the observability condition \cite{MH12}. Also like the situation described by last subsection,  dynamical regimes with more accuracy, $h=0.05$ and $0.1$, the chance for catastrophic filter divergence is higher. The application of AACI is important here to guarantees an accurate filter to work. The nonlinear dependence of explosion rate over $h$ was also observed in \cite{GM13}. 

\subsection{Dependence on the numerical integrator}
\label{sec:integrator}\label{s:num_integrators}
In this section, we address the question of whether the catastrophic filter divergence can be avoided by simply using a more stable integrator in EnKF and EnKF-CI. We will see that with implicit methods and variable time step methods, one can avoid catastrophic filter divergence, but at prohibitive computational cost. Moreover, although they avoid catastrophic filter divergence, they do not avoid ordinary filter divergence; in particular we will see that EnKF and EnKF-CI with implicit / variable time-steps do not exhibit any filtering skill, in contrast with EnKF-AI and EnKF-CAI. This corroborates the findings of \cite[Chapter 2]{MH08}, where it was shown that numerical model error from implicit methods leads to reduced filtering skill.     
\par
%As mentioned earlier, the experiments above are carried out using the explicit Euler scheme with time step $\Delta=10^{-4}$. Although this is an innappropriate choice of integrator for a stiff system, we select this integrator to illustrate that it can be used to considerable skill, provided an adaptive inflation is included. In this section we compare the performance of the cheap AACI (and C+AACI) method using explicit Euler against the non-adaptive EnKF (and CCI), with more appropriate (and more expensive) integrators for stiff systems.  
%\par
In Table \ref{tab:intEnKF}, we display the results from $100$ trials, comparing the performance of EnKF-AI with explicit Euler (step size $\Delta = 10^{-4}$) against EnKF with (respectively) explicit Euler ($\Delta = 10^{-4}$), standard $4$-th order Runge-Kutta ($\Delta = 2.5 \times 10^{-3}$), {\tt{ode45}} and implicit Euler ($\Delta =  10^{-2}$). {\tt ode45} is a native MATLAB integrator, which is based on a Runge-Kutta $(4,5)$ method with variable step size. The implicit equation in implicit Euler is solved using the native MATLAB function {\tt{fslove}}. The average time is also shown to indicate the computational expense associated with each method. 
\par
We see that {\tt{ode45}} and implicit Euler are the only methods that do not exhibit catastrophic filter divergence, but each trial is enormously expensive when compared to the explicit Euler based methods. Moreover, {\tt{ode45}} and Implicit Euler display zero filtering skill, as judged by the RMSE and pattern correlation (the benchmark RMSE is $12.93$). As discussed earlier, EnKF-AI also exhibits little skill in this regime, but we shall see that the skill is dramatically improved by introducing additional constant inflation. Unsurprisingly, the explicit Euler based methods are much faster than both the higher order and the implicit methods and the EnKF-AI method is only fractionally more expensive than EnKF.     
\begin{table}[h!]
\begin{center}
\begin{tabular}{|c | c| c c c c|}
\hline
 & EnKF-AI & & EnKF  & & \\
\hline
Integrator & Explicit & Explicit & RK4 &{\tt{ode45}}  & Implicit\\
\hline
Cata. Div. &0\% &100\% &92\% &0\% & 0\% \\
RMSE & 24.48 & NaN & NaN  &44.66 & 20.38 \\
Pattern Cor. & 0.23 & NaN & NaN &0.22 & 0.11\\
Avg. Time  &2.49 &2.39 &19.53 &41.00 & 662.48\\
\hline
\end{tabular}
\end{center}
\caption{The performance and time cost of EnKF-AI (explicit Euler, $\Delta = 10^{-4}$)) against EnKF in the $F=16$ regime. EnKF is implemented by explicit Euler ($\Delta = 10^{-4}$), RK4 ($\Delta = 2.5\times 10^{-3}$), {\tt{ode45}} and implicit Euler ($\Delta =  10^{-2}$). The benchmark RMSE is $12.93$. The time cost is measured in seconds.  }
\label{tab:intEnKF}
\end{table}

In Table \ref{tab:intCCI} we display the results from a similar experiment, now comparing EnKF-CAI (with explicit Euler) against EnKF-CI with the same four choices of integrator. Catastrophic filter divergence is overall less prevalent, with {\tt{ode45}} and implicit Euler once again exhibiting no catastrophic filter divergence. In contrast with the previous experiment, EnKF-CAI exhibits good filtering skill as judged by the RMSE and pattern correlation (once again, the benchmark RMSE is $12.93$). EnKF-CI with {\tt{ode45}} exhibits less filtering skill than EnKF-CAI, with RMSE not beating the benchmark and EnKF-CI with implicit Euler exhibits even less skill. As in the first experiment, the adaptive inflation is more accurate and far cheaper than the higher order and implicit methods.   

\begin{table}[h!]
\begin{center}
\begin{tabular}{|c | c| c c c c|}
\hline
& EnKF-CAI & & EnKF-CI  & & \\
\hline
Integrator & Explicit & Explicit &RK4 &{\tt{ode45}}  & Implicit\\
\hline
Cata. Div. &0\% &18\% &5\% &0\% & 0\% \\
RMSE & 11.91 & NaN & NaN  &14.95 & 16.29 \\
Pattern Cor. & 0.69 & NaN & NaN &0.71 & 0.42\\
Avg. Time  &2.65 &2.34 &21.48 &37.02 & 650.31\\
\hline
\end{tabular}
\end{center}
\caption{The performance and time cost of EnKF-CAI (explicit Euler, $\Delta = 10^{-4}$)) against EnKF-CI in the $F=16$ regime. EnKF is implemented by explicit Euler ($\Delta = 10^{-4}$), RK4 ($\Delta = 2.5\times 10^{-3}$), {\tt{ode45}} and implicit Euler ($\Delta =  10^{-2}$). The benchmark RMSE is $12.93$. The time cost is measured in seconds.   }
\label{tab:intCCI}
\end{table}

\subsection{The distribution of $\Eo$ and $\Cu$}\label{s:num_dist}
The filter ensemble statistics $\Eo$ and $\Cu$, defined in Section \ref{s:aaci_algo} are used to determine when the adaptive inflation of EnKF-AI should be triggered. In particular, the mechanism is triggered whenever either of the two statistics passes an appropriately chosen threshold. In this section we look at the distribution of these statistics in order to evaluate the choice of thresholds introduced in Section \ref{s:choice}. We have collected the data of $\Eo$ and $\Cu$ from the EnKF-AI scheme from all the analysis steps. In Table \ref{tab:stats} we compare the distributions of $\Eo$ and $\Cu$ in the three different turbulent regimes and compare with the choice of threshold. 
\par 
It is clear from the table that stronger turbulence significantly increase the size of both $\Eo$ and $\Cu$. The thresholds $\Me$ and $\Mc$ are always larger than the average of $\Eo$ and $\Cu$, but the ratios $\Me/\langle \Eo\rangle$ and $\Mc/\langle \Cu\rangle$ significantly decreases as the turbulence gets stronger. As a result, the percentage of $\Eo$ and $\Cu$ that are beyond the thresholds significantly increases. Connecting this observation with the performance of EnKF-AI, which is good in $F=4$, slightly off the benchmark in $F=8$, and very bad in $F=16$, we can conclude that whether $\Eo$ and $\Cu$ are beyond the thresholds is a good indicator of the malfunctioning. As intended, the adaptive inflation is rarely triggered when the filter operates properly, but frequently triggered when the situation is chaotic. 
%
%$\Eo$ and $\Cu$ are two key statistics we used to turn on AACI and determine its strength. It is essential to understand their overall distribution in different regimes, since they can help us to evaluate how good are the thresholds. We have collected the data of $\Eo$ and $\Cu$ from the AACI scheme from all the analysis steps. The details are summarized by the following table and Figure 
\begin{table}[h!]
\begin{center}
\begin{tabular}{|c | c  c c|}
\hline
$F$ & 4 & 8 &16\\
\hline
$\Me$ & 32.5 &69.6 &127.6 \\
Average $\Eo$ & 2.95 & 20.28 & 70 \\
$\BP(\Eo>\Me)$ & 0.0285\% & 3.025\% & 9.73\% \\
\hline
\hline
$\Mc$ & 6.2 &28.8 &81.4 \\
Average $ \Cu$ & 0.001 & 0.0437 & 0.22 \\
 $\BP(\Cu>\Mc)$ & 0\% & 0\% & 0\% \\
\hline
\end{tabular}
\end{center}
\caption{For the three turbulent regimes $(F= 4,8,16)$, we list the thresholds $\Me$ and $\Mc$ chosen according to the method of Section \ref{s:choice}. We also list the average of $\Eo$ and $\Cu$ over $100$ independent trials and also give the probability that the statistics pass their respective threshold.  }
\label{tab:stats}
\end{table}

The histograms of $\Eo$ and $\Cu$ in three regimes are presented in Figure \ref{fig:tail}. $\Eo$ has a Gaussian like tail in the weak turbulence regime and exponential like in strong turbulence. The tail of $\Cu$ is much heavier and polynomial-like. The thresholds we compute from the benchmarks are fairly large comparing with the distribution. These two posterior observations suggest that using even an aggressive threshold strategy as \ref{threshold:aggressive}, the outliers are relatively scarce so adaptive inflation does not hamper the performance of the filters.
\par
It is worth noting that $\Cu$ is usually a small number and very rarely exceeds the threshold $\Mc$. This suggests that the filter performance would be unaffected if the adaptive inflation was changed to $
\lambda_n=c_\cut \Eo $ so that only $\Eo$ is used to trigger the inflation. We have tested this simpler adaptive inflation mechanism with both EnKF-AI and EnKF-CAI in the setting of the previous subsections, and the performances are much as the same. From a theoretical perspective however, both statistics must be included to obtain a rigorously stable filter.

%On the other hand, in the $F=16$ case, the maximum of $\Cu$ is $110.91$. This suggests that $\Cu$ probably is a unbounded variable, so it is necessary to include it in the AACI scheme in order for theoretical stability results to hold. 
%
 \begin{figure}[!h]
\hspace{-2cm}\scalebox{0.5}{\includegraphics{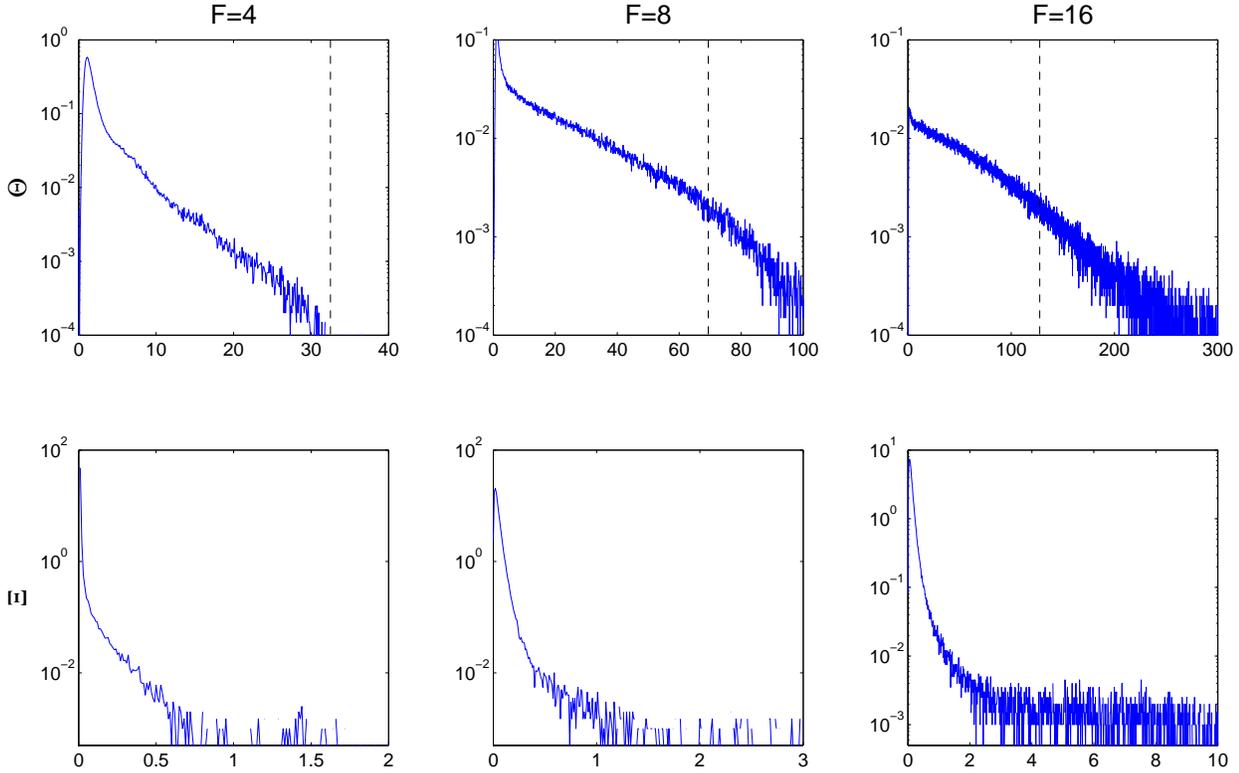}}
\caption{Histograms of $\Eo$ (upper panel) and $\Cu$ (lower panel) in $F=4,8,16$ regimes. The thresholds $\Me$ are marked as black dashes. The thresholds $\Mc$ lie out side the $x$-axis for $\Cu$. Note the logarithmic scale on the vertical axis. 
\label{fig:tail}}
\end{figure}
\subsection{Discussion of Experiments}
In Section \ref{s:num_turb}, the numerical simulations convincingly demonstrate that catastrophic filter divergence is prevalent in EnKF and EnKF-CI, particularly when using an unstable integrator in the forecast step. By comparing with EnKF-AI and EnKF-CAI, adaptive inflation is shown to successfully avoid catastrophic filter divergence, without sacrificing any filtering skill. In the weakly turbulent regime, the adaptive inflation mechanism is rarely turned on and EnKF-AI  performs similarly to EnKF and likewise for the constant inflation counterparts. In more turbulent regimes, the adaptive inflation is regularly triggered and the adaptive inflation filters display filtering skill, typically beating the RMSE benchmark. In Section \ref{s:num_tuning} it is also shown that this skill can be optimized by tuning constant inflation strength.    

%
%a common phenomenon in ensemble based Kalman filters, and the AACI procedure is a simple and effective way to ensure stability without sacrificing accuracy. The numerical evidence validates our theoretical claims in Theorem \ref{thm:adEnKF}. Moreover, the trajectories of the ensemble members  have recurrent pattern among all dynamical regimes, which validates the claim of geometric ergodicity statement of EnKF with AACI given in Theorem \ref{thm:geometric}. 

%\par
%From the perspective of accuracy, the adaptive inflation mechanism improves or at least preserves the accuracy of the original filters. When filtering a weak turbulence with $F=4$, the adaptive inflation is rarely triggered, hence the accuracy is preserved to the level of the original filter.  When filtering  moderate or strong turbulence, the adaptive mechanism is frequently triggered to avoid catastrophic filter divergence. In the strongly turbulent regime, the adaptive+constant inflation filter yields the most accurate results, but to consistently beat the benchmark accuracy score, the inflation strength $\rho$ and (if possible) the observation interval $h$, should be appropriately tuned. 
\par
We saw in Section \ref{s:num_integrators} that catastrophic filter divergence can be avoided by using implicit or variable time step integrators, but at the expense of prohibitive computational cost and the loss of filtering skill. Indeed, one would be better off using the naive estimator developed for the RMSE benchmark. 
\par
The adaptive inflation method permits for the effective use of very cheap integrators such as explicit Euler in the forecast step. As evidenced by EnKF, it is impossible to use such integrators in turbulent models, due to prevalence of catastrophic filter divergence. The adaptive inflation completely avoids this issue.

%
%without this adaptive inflation explicit integrators will be susceptible to instability and hence catastrophic filter divergence. Adaptive methods never suffer from such instability and moreover is able to beat non-adaptive methods with high-order and implicit integrators used in the forecast step.  
%
%
%The improvement over accuracy may help the filter to surpass the benchmark, like in the case of $F=16$ with CCI. But overall, the scale of improvement for AACI is not big, especially comparing with difference can made by tuning the constant covariance inflation strength. This reveals that the accuracy of EnKF is not a consequence of stability, but rather a more delicate issue. On the other hand, AACI can significantly backup CCI, as one can choose the optimal inflation strength without the danger of catastrophic filter divergence. \todo{DK: not relevant anymore?}
%\par
%The thresholds created by Kalman one time assimilation is much sharper than the ones given by the equilibrium mean estimation in the weak turbulence regime, but lose their sharpness when the turbulence grows stronger. This is much due to the bound created by Kalman one time assimilation grows exponentially with respect to the expansion rate $\beta_L$, and hence also with respect to $|F|$, as seen in Lemma \ref{lem:ODE}. This may be used as an a-priori dynamical criterion of filter's performance, considering the fact that the filters perform rather poorly with medium and strong turbulence. 
%\par

In Section \ref{s:num_dist}, we provide the distribution analysis of $\Eo$ and $\Cu$ which indicates the thresholds are achieving our purpose. We also saw that the threshold for $\Cu$ was never triggered, indicating that a simpler adaptive inflation only using the statistic $\Eo$ will have practically the same performance. This has been validated, but is not presented here. 
%
%If we see EnKF in the $F=4$ regime as performing properly, the actual false detection rate is $3$ times among $NT/2h=10^6$ iterations, which is extremely low. On the contrary, the adaptive inflation mechanism is turned on frequently in the strong turbulent regimes. In conclusion, $\Eo$ and $\Cu$ can be used as indicators for malfunctioning of the filters. The fact that $\Cu$ is relatively small indicates that a simpler covariance inflation with only $\Eo$ is practically feasible. This is tested but not presented here. 
%\par
%Lastly, we review the benchmarks we constructed based on statistics we collected from AACI. From the logarithmic histogram,  the tail of $\Eo$ is approximately  an exponential one. Its standard deviation is $\langle\Eo\rangle=57.23$. The percentage lying outside $\sigma_{\Eo,A}$ is $9.66\%$, and there are  $0.12\%$ outliers beyond  $3\sigma_{\Eo,A}$. From the logarithmic histogram,  $\Cu$ has a polynomial tail. Its average is $\langle \Cu\rangle =0.25$, which is far less than the bound we constructed. However, the tail is fairly fat. The percentage lying outside $M_{\Cu,A}$ is $0.002\%$, with the maximum being $62.91$, indicating there are occasional alignments that creates strong cross covariance. 

\section{Conclusion and Discussion}\label{s:conclusion}

In this article we have proposed a novel modification of ensemble based methods using an adaptive covariance inflation. This modified method resolves common issues with ensemble based methods, which were raised in the introduction, namely that there is a lack of understanding of the stability properties of EnKF, with catastrophic filter divergence drawing stark attention to this fact. The EnKF-AI and related filtering methods (ETKF-AI, EAKF-AI, EnKF-CAI, etc.) resolve this stability issue entirely. 
\par
In particular we can develop a complete stability theory for these filters. In Section \ref{s:stab} we have proved time-uniform mean-square upper bounds for EnKF-AI and related filters, using a Lyapunov argument. This is a considerable improvement on the known results for EnKF \cite{TMK15non} since the observable energy criterion is no longer required and the resulting Lyapunov function for the filter always has compact sub-level sets. In Section \ref{s:ergo} we exploit this fact to prove geometric ergodicity of the signal-ensemble process for EnKF-AI and related methods. The ergodicity result is a type of stability which ensures that filter initialization errors are dissipated exponentially quickly by the method.     
\par
In Section \ref{s:numerics} we provide numerical evidence which supports the stability theory developed for EnKF-AI and related methods. In particular, catastrophic filter divergence is completely eliminated by adaptive inflation, even when cheap unstable integrators are used in the forecast step. For such integrators, catastrophic filter divergence is unavoidable in EnKF and EnFK-CI. Moreover, adaptive inflation does not sacrifice any filtering skill, with EnKF-AI and EnKF-CAI displaying good accuracy when compared to benchmark accuracy scores, even in highly turbulent regimes. This suggests that adaptive inflation may in fact make filtering algorithms faster, since they allow for such cheap, typically unstable forecast integrators.

In recent years, there have been  many methods introduced featuring a type of adaptive inflation filtering. These methods are motivated by very different principles. For example, \cite{And07, And09} introduce a method that uses the norm of innovation to filter the optimal parameter for multiplicative inflation. \cite{BS13,ZH15} introduce two different ways to find the covariances of the system and observation noises, assuming they are not known. \cite{YZ15} designs a covariance relaxation method, so the posterior ensemble is pulled back to the forecast ensemble with a linear factor obtained from the innovation process. In \cite{sanzstuart14} a modification of 3DVAR is proposed which relies on projection back to a stable region, this modification has been theoretically analyzed from the perspective of long-time accuracy. On one hand, these methods are quite different from the EnKF-AI method we introduced. With the exception of \cite{sanzstuart14}, none have been rigorously proved to be stable in the sense of our theorems here. On the other hand, it is very possible that our theoretical framework here can be extended to these methods to build a stability framework, at least in some simple settings. The compelling reason here is that all these methods use the innovation sequence to decide the inflation, relaxation or noise covariance, which all grow with respect to $\Eo$. From this perspective, our framework here offers an explanation of the good performance of these data assimilation methods, as they are more resilient against filter divergence; and when the filter is on the edge of malfunctioning, the adaptive mechanisms can pull the ensemble back to more reasonable states. 
\par

\appendix

\section{Elementary claims}
%\begin{lem}
%\label{lem:inver}
%Let $A$ be a positive semidefinite symmetric matrix, then the following holds:
%\[
%0\preceq A(A+I)^{-1}\preceq I,\quad 0\preceq (A+I)^{-1}\preceq I,\quad 0\preceq (A+I)^{-1}\preceq A^{-1}
%\]
%\end{lem}
%\begin{proof}
%Since $A(A+I)^{-1}+(A+I)^{-1}=I$, it suffices to show $0\preceq (A+I)^{-1}\preceq I$. Since $A$ is positive semidefinite and symmetric, it can be diagonalized through an orthogonal matrix $\Psi,$ i.e. $A=\Psi D \Psi^T$ with $D$ being diagonal. Then based on $(A+I)^{-1}=\Psi (D+I)^{-1}\Psi^T$, it is elementary to conclude our lemma. 
%\end{proof}
\begin{lem}
\label{lem:young}
By Young's inequality, for any $\epsilon>0$ and $x,y\in \mathbb{R}^d$, the following holds:
\[
|x+y|^2=|x|^2+|y|^2+2\langle x, y\rangle\leq (1+\epsilon^2) |x|^2+(1+\epsilon^{-2})|y|^2. 
\]
\end{lem}
%
%\begin{lem}
%\label{lem:trinv}
%For any two $d\times d$ positive semidefinite symmetric matrices $A$ and $B$, the following holds:
%\[
%\text{tr}(A(A+B)^{-1})=\text{tr}((A+B)^{-1}A)
%\leq \text{tr}((A+B)^{-1}(A+B))= d. 
%\]
%\end{lem}
\begin{lem}
\label{lem:mean}
For any group of vectors $v_1,\ldots, v_n$, the following holds:
\[
\arg\min_{v} \sum_{k=1}^n|v_i-v|^2=\frac{1}{n}\sum_{k=1}^nv_i.
\]
\end{lem}
\begin{proof}
One can verify this by using the KKT condition. 
\end{proof}
%
%\begin{lem}
%\[
%S=\widehat{S}[I_K+(K-1)^{-1}\widehat{S}^TH^TH\widehat{S}]^{-1/2}\Theta(\vec{F})[\widehat{S}\widehat{S}^T]^{-1/2}\widehat{S}
%\]
%\end{lem}
%\begin{proof}
%Let $N$ be a matrix with its columns form an orthogonal basis of the null space of $\widehat{S}$, i.e. 
%\[
%N=[e_1,\ldots, e_u], \quad \widehat{S}N=0,\quad n+\text{rank}(\vec{F})=d. 
%\] 
%Then we can find $F$ and $X$ such that
%\[
%\widehat{S}\widehat{S}^T=[N,F]\begin{bmatrix}0&0\\ 0 &\Sigma \end{bmatrix}\begin{bmatrix}N^T\\F^T\end{bmatrix},\quad
%I_K+(K-1)^{-1}\widehat{S}^TH^TH\widehat{S}=[N,F]\begin{bmatrix}0&0\\ 0 &\Sigma \end{bmatrix}\begin{bmatrix}N^T\\F^T\end{bmatrix}
%\]
%\end{proof}
\section{Stability for ETKF-AI, EAKF-AI}
\begin{proof}[Proof for Theorem \ref{thm:adESRF}]
We use the observed-unobserved notation introduced in Section \ref{s:enkf_algo}. Notice that the mean update of \eqref{sys:ESQFad} is the ensemble averaged version of \eqref{sys:EnKFad} without the artificial perturbations $\xi_n^{(k)}$. Applying Jensen's inequality
\[
\mathbb{E}_{n-1}|\overline{X}_n|^2\leq \frac{1}{K}\sum_{k=1}^K\mathbb{E}_{n-1}|X_n^{(k)}|^2
\]
to \eqref{tmp:xk}, we find the following with a constant $D_1$,
\[
\mathbb{E}_{n-1}|\overline{X}_n|^2
-2\rhom^{-1}\|H_0\|^2\mathbb{E}_{n-1}|HU_n|^2\leq D_1.
\]
The corresponding part for \eqref{tmp:Yj2} will become
\[
|\overline{Y}_n|^2\leq (1+\frac{1}{2}\beta_h)|\overline{\Yhat}_n|^2+(1+2\beta_h^{-1})|\widehat{B}^T_nH_0^T(I+H_0\widetilde{C}^X_nH_0^T)^{-1}(H_0\overline{\Xhat}_n-Z_n)|^2. 
\]
Applying Jensen's inequality to \eqref{tmp:bh}, we find the second part above is bounded by
\[
|\widehat{B}^T_nH_0^T(I+H_0\widetilde{C}^X_nH_0^T)^{-1}(H_0\overline{\Xhat}_n-Z_n)|^2
\leq K\|H_0\|^2\max\{M^2_1\Mc^2, \rhom^{-2}c_\cut^{-2}\}. 
\]
So  there is a constant $D_2$ such that 
\[
\mathbb{E}_{n-1}|\Vbar_n|^2-2\rhom^{-1}\|H\|^2\mathbb{E}_{n-1}|HU_n|^2\leq (1+\frac{1}{2}\beta_h)|\overline{\Vhat}_n|^2+D_2,
\]
where we use the fact that $\|H_0\|^2=\|H\|^2$. Then notice that the posterior covariance follows $C_n\preceq\Chat_n$ by \eqref{eqn:postprior}. Therefore using that 
\[
\sum_{k=1}^K|V_n^{(k)}|^2=K|\Vbar_n|^2+\text{tr}(C_n)\leq 
K|\Vbar_n|^2+\text{tr}(\Chat_n),
\]
and $\sum_{k=1}^K|\Vhat_n^{(k)}|^2=K|\overline{\Vhat}_n|^2+\text{tr}(\Chat_n)$, we obtain from the previous inequality 
\[
\mathbb{E}_{n-1}\sum_{k=1}^K|V_n^{(k)}|^2-2K\rhom^{-1}\|H\|^2\mathbb{E}_{n-1}|HU_n|^2
\leq (1+\frac{1}{2}\beta_h)\mathbb{E}_{n-1}\sum_{k=1}^K|\Vhat_n^{(k)}|^2+D_2.
\]
Assumption \ref{aspt:kinetic} leads to 
\[
(1+\frac{1}{2}\beta_h)\mathbb{E}_{n-1}\sum_{k=1}^K|\Vhat_n^{(k)}|^2
\leq (1-\frac{1}{2}\beta_h)\sum_{k=1}^K|V_{n-1}^{(k)}|^2+(1+\frac{1}{2}\beta_h)K K_h,
\]
and moreover
\[
\mathbb{E}_{n-1}\sum_{k=1}^K|V_n^{(k)}|^2-2K\rhom^{-1}\|H\|^2\mathbb{E}_{n-1}|HU_n|^2
\leq (1-\frac{1}{2}\beta_h)\sum_{k=1}^K|V_{n-1}^{(k)}|^2+(1+\frac{1}{2}\beta_h)K K_h+D_2.
\]
Then notice a multiple of Assumption \ref{aspt:kinetic} is 
\[
4K\beta_h^{-1}\rho_0^{-1}\|H\|^2\mathbb{E}_{n-1}|U_n|^2\leq 4K\beta_h^{-1}\rho_0^{-1}\|H\|^2(1-\beta_h)|U_{n-1}|^2+4K\beta_h^{-1}\rho_0^{-1}\|H\|^2K_h. 
\]
The sum of the two previous two inequalities yields 
\begin{align*}
\mathbb{E}_{n-1}\CE_n&=\sum_{k=1}^K\mathbb{E}_{n-1}|V^{(k)}_n|^2+
4K\beta_h^{-1}\rho_0^{-1}(1-\frac{1}{2}\beta_h)\|H\|^2\mathbb{E}_{n-1}|U_n|^2\\
&\leq (1-\frac{1}{2}\beta_h)\sum_{k=1}^K|V_{n-1}^{(k)}|^2+ 4K\beta_h^{-1}\rho_0^{-1}\|H\|^2(1-\beta_h)|U_{n-1}|^2+D,
\end{align*}
with $D:=4K\beta_h^{-1}\rho_0^{-1}\|H\|^2K_h+(1+\frac{1}{2}\beta_h)K K_h+D_2$. Then notice that $(1-\frac{1}{2}\beta_h)^2\geq (1-\beta_h)$, therefore
\[
\mathbb{E}_{n-1}\CE_n
\leq (1-\frac{1}{2}\beta_h)\CE_{n-1}+D. 
\]
\end{proof}

\section{Ergodicity formulas}\label{app:ergo_formulas}
Here we present the concrete formulas that are used in Section \ref{s:ergo} for ETKF-AI and EAKF-AI. We omit unneccessary details concerning the transform and adjustment matrices, which can be  found in \cite{bishop01, And01,MH12,TMK15non}. In both ESRF methods, the Markov kernel 
$\Phi : \CX \times \CB(\CY) \to [0,1]$ with $\CE:=\reals^d \times \reals^{d\times K}$ and $\CY : = \reals^d \times \reals^{d \times K} \times \reals^q$ is described by 
\begin{equ}
(U_{n-1} , V_{n-1}^{(1)},\dots,V_{n-1}^{(K)}) \mapsto (U_n, \Vhat_{n}^{(1)},\dots,\Vhat_{n}^{(K)},Z_n)\;.
\end{equ}
The deterministic step is given by the map  $\Gamma(U , V , Z ) = (U,\Gamma^{(1)},\dots,\Gamma^{(K)})$ where
\begin{equation}
\label{e:gammak}
\Gamma^{(k)} = \Vhatbar - \widetilde{C} H^T (I + H \widetilde{C} H^T)^{-1} (H \Vhatbar - Z)  + S^{(k)}
\end{equation}
with $\Vbar = \frac{1}{K}\sum_{k=1}^K \Vhat^{(k)}$ and $\widetilde{C}=\Chat = \frac{1}{K-1}\sum_{k=1}^K (\Vhat^{(k)} - \Vbar) \otimes (\Vhat^{(k)} - \Vbar)  $ and $S^{(k)}$ is the $k$-th column of the updated spread matrix $S$. $S$ is given by $\Shat T(\Shat)$ in the ETKF method, and $A(\Shat)\Shat$ in the EAKF method, where $\Shat$ is the forecast spread matrix $\Shat = (\Vhat^{(1)} - \Vhatbar , \dots ,\Vhat^{(K)} - \Vhatbar)$.
\par
With the ETKF method, there are a few way to define the transformation matrix $T(\Shat)$, one reasonable choice is taking the matrix square root
\begin{equ}
T(\Shat)=\big(I_K+(K-1)^{-1}\widehat{S}^TH^TH\widehat{S}\big)^{-1/2}=\big(I_K-(K-1)^{-1}\widehat{S}^TH^T(I+H\widehat{C}H^T)^{-1}H\widehat{S}^T\big)^{1/2}\;.
\end{equ} 
\par
As for the EAKF method, the construction of the adjustment matrix $A(\Shat)$ is slightly more complicated
as the following 
\begin{equ}
A(\Shat)=Q\Lambda G^T(I+D)^{-1/2}\Lambda^{\dagger}Q^T .
\end{equ}
Here $Q\Lambda R$ is the SVD decomposition of $\Shat$ and $G^TDG$ is the diagonalization of $(K-1)^{-1}\Lambda^T Q^T H^TH^TQ\Lambda^T$, and $\dagger$ indicates pseudo inverse of a matrix. 
\par
And for the ETKF-AI and EAKF-AI, everything is the same as their counterpart without adaptive inflation, except that we use the following in \eqref{e:gammak}
\begin{equ}
 \widetilde{C} = \frac{1}{K-1}\sum_{k= 1}^{K}(\widehat{V}^{(k)}-\overline{\widehat{V}})\otimes (\widehat{V}^{(k)}-\overline{\widehat{V}}) + c_\varphi\Eo (1+\Cu)\unit_{\Eo > \Me {\rm \;or\;} \Cu > \Mc} I .
\end{equ}

\bibliographystyle{unsrt}
\bibliography{aci}

\end{document}